\documentclass[11pt]{article}       
\usepackage{geometry}               
\usepackage{graphicx}
\usepackage{caption}
\usepackage{amsmath} 
\usepackage{amssymb}
\usepackage{amsthm}
\usepackage{bm}
\usepackage{lipsum}
\usepackage[linesnumbered, ruled]{algorithm2e}
\usepackage{enumerate}
\usepackage{color}
\usepackage{array}
\usepackage{cite}
\usepackage{setspace}
\usepackage{url}
\usepackage{cite}

\newtheorem{proposition}{Proposition}

\newtheorem{corollary}{Corollary}
\newtheorem{lemma}{Lemma}
\newtheorem{theorem}{Theorem}
\newtheorem{remark}{Remark}
\newtheorem{assumption}{Assumption}

\numberwithin{equation}{section}

\DeclareMathOperator*{\argmin}{arg\,min}

\title{A Convex Optimization Approach to Dynamic Programming
 in Continuous State and Action Spaces} 

\author{
 Insoon Yang\thanks{Department of Electrical and Computer Engineering, Automation and Systems Research Institute,  Seoul National University ({insoonyang@snu.ac.kr}). This work was supported in part by  the Creative-Pioneering Researchers Program through SNU, the National Research Foundation of Korea funded by the MSIT(2020R1C1C1009766), and Samsung Electronics.}
}

\date{}
\providecommand{\keywords}[1]{\textbf{Key words.} #1}

\begin{document}
\maketitle

\pagestyle{myheadings}
\thispagestyle{plain}

\begin{abstract}
In this paper, a convex optimization-based method is proposed for numerically solving dynamic programs in continuous state and action spaces.
The key idea is to approximate the output of the Bellman operator at a particular state by the optimal value of a convex program.
The approximate Bellman operator has a computational advantage because it involves a convex optimization problem in the case of control-affine systems and convex costs. Using this feature, we propose a simple dynamic programming algorithm to evaluate the approximate value function at  pre-specified grid points by solving convex optimization problems in each iteration.
We show that the proposed method approximates the optimal value function with a uniform convergence property in the case of convex optimal value functions.
We also propose an \emph{interpolation-free} design method for a control policy, of which performance converges uniformly to the optimum as the grid resolution becomes finer.
When a nonlinear control-affine system is considered, 
the convex optimization approach provides an approximate  policy with a provable suboptimality bound. 
For general cases,
 the proposed convex formulation of dynamic programming operators can be  modified as a nonconvex bi-level program, in which the inner problem is a linear program, without losing uniform convergence properties.
\end{abstract}


\section{Introduction}

Dynamic programming (DP) has been one of the most important methods for solving  and analyzing sequential decision-making problems in optimal control, dynamic games, and reinforcement learning; among others.
By using DP,
we can decompose a complicated sequential decision-making problem into multiple tractable subproblems, of which optimal solutions are used to construct an optimal policy of the original problem. 
Numerical methods for DP are the most well
studied for discrete-time Markov decision processes (MDPs) with discrete state and action spaces~(e.g., \cite{Puterman2014}) and continuous-time deterministic and stochastic optimal control problems in continuous state spaces~(e.g., \cite{Kushner2013}). 
This paper focuses on the discrete-time case with continuous state and action spaces in  the  finite-horizon setting.
Unlike infinite-dimensional linear programming (LP) methods~(e.g., \cite{Hernandez2012, Savorgnan2009, Dufour2013}), which require a finite-dimensional approximation of the LP problems,
 we 
  develop a finite-dimensional convex optimization-based method, that uses a discretization of the state space, while not discretizing the action space. 
Moreover, we assume that a system model and a cost function are explicitly known unlike the literature on reinforcement learning (RL)~(e.g., \cite{Sutton2018, Bertsekas2019, Szepesvari2010} and the references therein). 
 Note that our focus is not to resolve the scalability issue in DP.
As opposed to the RL algorithms that seek approximate solutions to possibly high-dimensional problems~(e.g., \cite{Mnih2015, Schulman2015, Lilicrap2015, Haarnoja2018}), our method is useful, when a provable convergence guarantee is needed
for problems with relatively low-dimensional state spaces.\footnote{However, our method is suitable for problems with high-dimensional action spaces.}

Several discretization methods have been developed
for discrete-time DP problems in continuous (Borel) state and action spaces. These methods can be assigned to two categories.  
The first category discretizes both state and action spaces.
Bertsekas~\cite{Bertsekas1975} proposed two discretization methods 
and proved their convergence under a set of assumptions, including Lipschitz type continuity conditions.
Langen~\cite{Langen1981} studied the weak convergence of  an approximation procedure, although no explicit error bound was provided.
However, the discretization method, proposed by Whitt~\cite{Whitt1978} and Hinderer~\cite{Hinderer1978}, is shown to be convergent and to have error bounds. Unfortunately, these error bounds are sensitive to the choice of partitions, and additional compactness and continuity assumptions are needed to reduce the sensitivity. 
Chow and Tsitsiklis~\cite{Chow1991} developed a multi-grid algorithm, which could be more efficient than its single-grid counterpart in achieving a desired level of accuracy.
The discretization procedure, proposed by 
Dufour and Prieto-Rumeau~\cite{Dufour2012},  can handle
 unbounded cost functions and locally compact state spaces.
 However, it still requires the Lipschitz continuity of some components of dynamic programs.
 A measure concentration result for the Wasserstein metric has also been used to measure the accuracy of the approximation method in~\cite{Dufour2015} for Markov control problems  in the average cost setting.
 Unlike the aforementioned approaches, 
the finite-state and finite-action approximation method for MDPs with $\sigma$-compact state spaces, proposed by
 Saldi {\it et al.}, does not rely on any Lipschitz-type continuity conditions~\cite{Saldi2017}.

The second category of discretization methods uses a computational grid only for the state space, i.e., this class of methods does not discretize the action space.
These approaches have a computational advantage over the methods in the first category, particularly when the action space dimension is large. 
The state space discretization procedures, proposed by Hern\'{a}ndez-Lerma~\cite{Hernandez1989}, are shown to have a convergence guarantee with an error bound under Lipschitz continuity conditions on elements of control models. However, they are subject to the issue of local optimality in solving the nonconvex optimization problem over the action space involved in the Bellman equation. 
Johnson {\it et al.}~\cite{Johnson1993} suggested spline interpolation methods, which are computationally efficient in high-dimensional state and action spaces.
Unfortunately, these spline-based approaches do not have a convergence guarantee or an explicit suboptimality bound.
Furthermore, the Bellman equation approximated by these methods involves nonconvex optimization problems.

The method proposed in this paper is classified into the second category of discretization procedures. 
Specifically, our approach only discretizes the state space, and thus it can handle high-dimensional action spaces. 
The key idea is to use an auxiliary optimization variable that assigns the contribution of each grid point when evaluating the value function  at a particular state.
By doing so, we can avoid an explicit interpolation of the optimal value function and control policies evaluated at the pre-specified grid points in the state space, unlike most existing methods.
The contributions of this work are threefold.
First, we propose an approximate version of the Bellman operator and show that the corresponding approximate value function 
 converges uniformly to the optimal value function $v^\star$  when $v^\star$ is convex. 
The proposed approximate Bellman operator has a computational advantage because it
involves a convex optimization problem  in the case of linear systems and convex costs. 
Thus, we can construct a control policy, of which performance converges uniformly to the optimum,
 by solving $M$ convex programs in each iteration of the DP algorithm, where $M$ is the number of grid points required for the  desired accuracy.
Second, we show that the proposed convex optimization approach provides an approximate policy with a provable suboptimality bound in the case of control-affine systems. 
This error bound is useful when gauging the performance of the approximate policy relative to an optimal policy.
Third, we propose a modified version of our approximation method for general cases by localizing the effect of the auxiliary variable. 
The modified Bellman operator involves a nonconvex bi-level optimization problem wherein the inner problem is a linear program. 
We show that both the approximate value function and
the cost-to-go function of a policy obtained by this method converge uniformly to the optimal value function if a globally optimal solution to the nonconvex bilevel problem can be evaluated; related error bounds are also characterized.
From our convex formulation and analysis, we observe that convexity in DP  deserves  attention  as it can play a critical role in designing a tractable numerical method that provides a convergent solution.

The remainder of the paper is organized as follows.
In Section~\ref{sec:setup}, we introduce the setup for dynamic programming problems and our approximation method.
The convex optimization-based method is proposed  in Section~\ref{sec:conv}.
We show the uniform convergence properties of this method when the optimal value function is convex. 
In Section \ref{sec:nonconv},
we characterize an error bound for cases with control-affine systems and
present a modified bi-level version of our approximation method for general cases. 
Section \ref{sec:num} contains numerical experiments that demonstrate
the convergence property  of our method.

\section{The Setup}\label{sec:setup}

\subsection{\bf Notation}

Let $\mathbb{B}_b(\mathcal{X})$ denote the set of bounded measurable functions on $\mathcal{X}$, equipped with the sup norm $\| v \|_\infty := \sup_{\bm{x} \in \mathcal{X}} | v(\bm{x}) | < +\infty$.
Given a measurable function $w: \mathcal{X} \to \mathbb{R}$, 
let $\mathbb{B}_{w}(\mathcal{X})$ denote the set of measurable functions $v$ on $\mathcal{X}$ such that $\| v \|_w := \sup_{\bm{x} \in \mathcal{X}} (| v(\bm{x}) | / w(\bm{x}) ) < +\infty$.
Let $\mathbb{L}(\mathcal{X})$ denote the set of lower semicontinuous functions on $\mathcal{X}$.
Finally, we let $\mathbb{L}_b(\mathcal{X}) := \mathbb{L}(\mathcal{X}) \cap \mathbb{B}_b(\mathcal{X})$ and  $\mathbb{L}_w(\mathcal{X}) := \mathbb{L}(\mathcal{X}) \cap \mathbb{B}_w(\mathcal{X})$.

\subsection{\bf Dynamic Programming Problems}

Consider a discrete-time Markov control system of the form
\[
x_{t+1} = f(x_t, u_t, \xi_t),
\]
where $x_t \in \mathcal{X} \subseteq \mathbb{R}^{n}$ is the system state, and  $u_t \in \mathcal{U}(x_t) \subseteq \mathcal{U}  \subseteq \mathbb{R}^{m}$ is the control input.
The stochastic disturbance process $\{\xi_t \}_{t\geq 0}$, $\xi_t \in \Xi \subseteq \mathbb{R}^{l}$, is i.i.d. and defined on a standard filtered probability space $(\Omega, \mathcal{F}, \{ \mathcal{F}_t\}_{t \geq 0}, \mathbb{P} )$. 
The sets $\mathcal{X}$, $\mathcal{U}(x_t)$, $\mathcal{U}$ and $\Xi$ are assumed to be Borel sets, and
the function $f : \mathcal{X} \times \mathcal{U} \times \Xi \to \mathcal{X}$ is assumed to be measurable.

A \emph{history} up to stage $t$ is defined by $h_t :=(x_0,u_0, \ldots, x_{t-1}, u_{t-1}, x_t)$. 
Let $H_t$ be the set of histories up to stage $t$ and $\pi_t$ be a stochastic kernel from $H_t$ to $\mathcal{U}$. 
The set of admissible control policies is chosen as 
\[
\Pi := \{\pi = (\pi_0,  \ldots, \pi_{K-1} ) :  \pi_t(\mathcal{U}(x_t) | h_t) = 1 \; \forall h_t \in H_t \}.
\]
Our goal is to solve the following finite-horizon stochastic optimal control problem:
\begin{equation}\label{opt}
\inf_{\pi \in \Pi} \; \mathbb{E}^\pi \bigg [
\sum_{t=0}^{K-1} r(x_t, u_t) + q(x_K) \mid x_0 = \bm{x}
\bigg ],
\end{equation}
where $r: \mathcal{X} \times \mathcal{U} \to \mathbb{R}$ is a measurable stage-wise cost function, $q: \mathcal{X} \to \mathbb{R}$ is a measurable terminal cost function,
 and 
$\mathbb{E}^\pi$ represents the expectation taken with respect to the probability measure induced by a policy $\pi$.
Under the following standard assumption for the \emph{measurable selection condition}, there exists a \emph{deterministic Markov} policy, which is optimal~(e.g., \cite[Condition 3.3.3, Theorem 3.3.5]{Hernandez2012}) and \cite[Assumptions 8.5.1--8.5.3, Lemma 8.5.5]{Hernandez2012b}).
In other words, we can find an  optimal policy $\pi^\star$, which is deterministic and Markov, i.e.,
\begin{align*}
\pi^\star \in \Pi^{DM} := \{\pi = (\pi_0, \ldots, \pi_{K-1}) :  \: &\pi_t: \mathcal{X} \to \mathcal{U} \mbox{ is measurable},\\
 & \pi_t (\bm{x}) = \bm{u} \in \mathcal{U}(\bm{x}) \; \forall \bm{x} \in \mathcal{X}\}.
\end{align*}
\begin{assumption}[Semicontinuous model]\label{ass:sc}
Let \[
\mathcal{K}:= \{ (\bm{x}, \bm{u}) : \bm{x} \in \mathcal{X}, \bm{u} \in \mathcal{U}(\bm{x}) \}
\]
  be the set of admissible state-action pairs.
\begin{enumerate}
\item The control set $\mathcal{U}(\bm{x})$ is compact for each $\bm{x} \in \mathcal{X}$, and the multifunction $\bm{x} \mapsto \mathcal{U}(\bm{x})$ is upper semicontinuous;

\item The real-valued functions $r$ and $q$ are lower semicontinuous  on $\mathcal{K}$ and $\mathcal{X}$, respectively.
 
\item 
There exist nonnegative constants $\bar{r}$, $\bar{q}$ and $\beta$, with $\beta\geq 1$, and a continuous weight function $w: \mathcal{X} \to [1, \infty[$ such that  
 $|r(\bm{x}, \bm{u}) | \leq \bar{r} w (\bm{x})$,\\  $|q(\bm{x}) | \leq \bar{q} w (\bm{x})$, and
$\mathbb{E} \big [ w (f(\bm{x}, \bm{u}, \xi)) \big ] \leq \beta w (\bm{x})$ for all $(\bm{x}, \bm{u}) \in \mathcal{K}$;

\item The function $(\bm{x}, \bm{u}) \mapsto  f(\bm{x}, \bm{u}, \bm{\xi})$ is continuous on $\mathcal{K}$ for each $\bm{\xi} \in \Xi$.

\item  The support $\Xi$ is finite, i.e., 
$\Xi := \{\hat{\xi}^{[1]}, \ldots, \hat{\xi}^{[N]}\}$ for some $\hat{\xi}^{[s]}$'s in $\mathbb{R}^l$, $s=1, \ldots, N$.

\end{enumerate}
\end{assumption}

 Note that by \emph{3)}, \emph{4)} and \emph{5)}, $(\bm{x}, \bm{u}) \mapsto \mathbb{E} \big [ w (f(\bm{x}, \bm{u}, \xi)) \big ]$ is continuous on $\mathcal{K}$.
For any $v \in \mathbb{B}_w(\mathcal{X})$, let
\begin{equation}\label{T_opt}
(\mathcal{T}v)(\bm{x}) := \inf_{\bm{u} \in \mathcal{U}(\bm{x})} \;   r(\bm{x}, \bm{u}) + 
\mathbb{E} \big [ v (f(\bm{x}, \bm{u}, \xi)) \big ]  \quad \bm{x} \in \mathcal{X}.
\end{equation}
Under Assumption~\ref{ass:sc}, the dynamic programming (DP) operator $\mathcal{T}$ maps $\mathbb{L}_w(\mathcal{X})$ into itself by \cite[Lemma 8.5.5]{Hernandez2012b}.
Let $v_t^{opt}: \mathcal{X} \to \mathbb{R}$, $t=0, \ldots, K$, be defined by the following Bellman equation:
\[
v_t^{opt} := \mathcal{T} v_{t+1}^{opt}, \quad t = 0, \ldots, K-1
\]
with $v_K^{opt} \equiv q$. 
Under Assumption~\ref{ass:sc},  $v_t^{opt} \in \mathbb{L}_w(\mathcal{X})$, and 
\[
v_t^{opt}(\bm{x}) = \inf_{\pi \in \Pi} \mathbb{E}^\pi  \bigg [
\sum_{k=t}^{K-1} r(x_k, u_k) + q(x_K)| x_t = \bm{x}
\bigg ],
\]
 i.e., $v_t^{opt}$ is the optimal value function of~\eqref{opt}.

Given a measurable function $\varphi: \mathcal{X} \to \mathcal{U}$, let
\[
(\mathcal{T}^{\varphi} v) (\bm{x}) :=  r(\bm{x}, \varphi(\bm{x})) + \mathbb{E} [
v (f(\bm{x}, \varphi(\bm{x}), \xi))
].
\]
Under Assumption~\ref{ass:sc}, there exists a measurable function  $\pi_t^{opt}: \mathcal{X} \to \mathcal{U}$ such that $\pi_t^{opt}(\bm{x}) \in \mathcal{U}(\bm{x})$, and 
\[
v_t^{opt}(\bm{x}) = \mathcal{T}^{\pi_t^{opt}} v_{t+1}^{opt}(\bm{x})
\]
for all $\bm{x} \in \mathcal{X}$ and all $t = 0, \ldots, K-1$.
Then, the deterministic Markov policy $\pi^{opt}:= (\pi_0^{opt}, \ldots, \pi_{K-1}^{opt}) \in \Pi^{DM}$ is an optimal solution to \eqref{opt}
by the dynamic programming principle~\cite[Theorem 3.2.1]{Hernandez2012}.
 In practice, Assumption~\ref{ass:sc}-\emph{5)} can be relaxed
using sampling-based methods to a probability distribution with a continuous support.
Specifically, 
the proposed convex optimization approach can be used in conjunction with the sampling-based (or `empirical') method in~\cite{Rust1997, Munos2008, Haskell2020}.
The convergence of the sampling-based method has been shown under  restrictive conditions such as compact state spaces or finite action spaces. 
The validity of such a combination has been numerically tested in Section~\ref{num:lc}. 
Its extension using stochastic subgradient methods can be found in~\cite{Jang2019}.

\subsection{\bf State Space Discretization}\label{sec:dis}

\begin{figure}[tb]
\begin{center}
\includegraphics[width=2.5in]{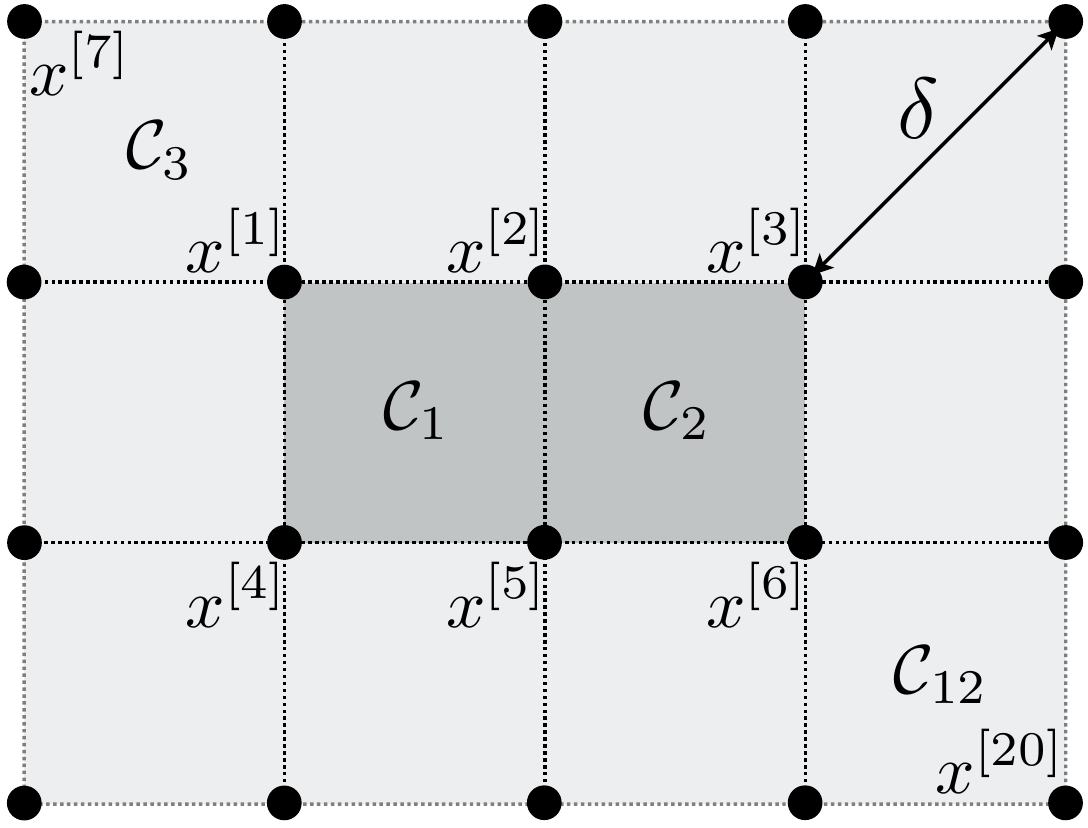}   
\caption{A two-dimensional example of the proposed discretization method. The dark gray box is $\mathcal{Z}_0 = \bigcup_{i=1}^2 \mathcal{C}_i$, and the light gray box is $\mathcal{Z}_1 = \bigcup_{i=1}^{12} \mathcal{C}_i$.  
The nodes $\bm{x}^{[1]}, \ldots, \bm{x}^{[20]}$ are selected as 
the vertices of the small square boxes such that $\bm{x}^{[1]}, \ldots, \bm{x}^{[6]} \in \mathcal{Z}_0$. In this example, $M_0 = 6$, and $M_1 = 20$.} \label{fig:grid}  
\end{center}           
\end{figure}

Our  method aims at approximating the optimal value function $v_t^{opt}$ at pre-specified nodes in a  convex compact set $\mathcal{Z}_t \subseteq \mathcal{X}$ for all~$t$.
We select a sequence $\{\mathcal{Z}_t\}_{t=0}^K$ of convex compact sets such that 
\[
\mathcal{Z}_{0} \subseteq \mathcal{Z}_{1} \subseteq \cdots \subseteq \mathcal{Z}_K \subseteq \mathcal{X},
\]
and
\begin{equation}\label{inclusion}
\big \{ f(\bm{x}, \bm{u}, \bm{\xi}) : \bm{x} \in \mathcal{Z}_{t}, \bm{u} \in \mathcal{U}(\bm{x}), \bm{\xi} \in \Xi \big \} \subseteq \mathcal{Z}_{t+1}.
\end{equation}
The existence of such a sequence is guaranteed under Assumption~\ref{ass:sc} because $f$ is continuous in $(\bm{x}, \bm{u})$, $\mathcal{U}(\bm{x})$ is compact, and $\Xi$ is finite.
Note that $\mathcal{Z}_K \neq \mathcal{X}$ in general, and  the state space $\mathcal{X}$ does not have to be compact.
Several examples satisfying these two conditions can be found in Section~\ref{sec:num}.
This choice of computational domains is also used in \cite{Bertsekas1975}.
With such a sequence, we can evaluate $v_t^{opt}$ on $\mathcal{Z}_t$ by using only information about $v_{t+1}^{opt}$ on $\mathcal{Z}_{t+1}$.

We choose $N_{\mathcal{C}}$ convex compact subsets $\mathcal{C}_1, \ldots, \mathcal{C}_{N_{\mathcal{C}}}$ of $\mathcal{Z}_K$ satisfying the following properties:
\begin{enumerate}
\item
 $\bigcup_{i=1}^{N_{\mathcal{C}}} \mathcal{C}_i = \mathcal{Z}_K$;
 \item $\mathcal{C}_i^o \cap \mathcal{C}_j^o = \emptyset$ for $i \neq j$, where $\mathcal{C}_i^o$ denotes the interior of $\mathcal{C}_i$;
 \item
For each $t$, there exists a subsequence $\{ \mathcal{C}_{i_j}\}_{j=1}^{N_{\mathcal{C}, t}}$ of $\{\mathcal{C}_i\}_{i=1}^{N_{\mathcal{C}}}$ such that\\
$\bigcup_{j=1}^{N_{\mathcal{C}, t}} \mathcal{C}_{i_j} = \mathcal{Z}_t$.
\end{enumerate}
We relabel $\mathcal{C}_i$, if necessary, so that 
$\bigcup_{i=1}^{N_{\mathcal{C}, t}} \mathcal{C}_i = \mathcal{Z}_t$
for all $t$. Note that \\ $N_{\mathcal{C}, K} = N_{\mathcal{C}}$, and $N_{\mathcal{C}, 0}  \leq N_{\mathcal{C}, 1} \leq \cdots \leq N_{\mathcal{C}, K}$.

Let $M_t$ denote the number of nodes (or grid points) in $\mathcal{Z}_t$.
The nodes $\bm{x}^{[1]}, \ldots, \bm{x}^{[M_K]} \in \mathcal{Z}_K$ are chosen such that each $\mathcal{C}_i$ is the convex hull of a \emph{subset} of the nodes. 
A concrete way to construct $\mathcal{Z}_t$'s and $\mathcal{C}_i$'s using a rectilinear grid is described in Appendix~\ref{app:grid}.
For example, in Fig. \ref{fig:grid}, $\mathcal{C}_1$ is the convex hull of $\{\bm{x}^{[1]}, \bm{x}^{[2]}, \bm{x}^{[4]}, \bm{x}^{[5]}\}$.
The maximum of the diameter of the sets $\{\mathcal{C}_i\}_{i=1}^{N_{\mathcal{C}}}$ is given by
\[
\delta := \max_{k=1, \ldots, N_{\mathcal{C}}} \max_{\bm{x}^{[i]}, \bm{x}^{[j]} \in \mathcal{C}_k} \| \bm{x}^{[i]} - \bm{x}^{[j]} \|.
\]
For example, in Fig. \ref{fig:grid}, $\delta$ is equal to the length of $\mathcal{C}_i$'s diagonal. 
We also relabel $\bm{x}^{[i]}$, if necessary, so that 
$\bm{x}^{[1]}, \ldots, \bm{x}^{[M_t]} \in \mathcal{Z}_t$ for all $t$.
Thus, 
\[
M_0 \leq M_1 \leq \cdots \leq M_K.
\]

\subsection{\bf Value Functions on Restricted Spaces}

Let $v_t^\star: \mathcal{Z}_t \to \mathbb{R}$ be
 the optimal value function restricted on $\mathcal{Z}_t$ for each $t$.
 In other words, 
 \[
 v_t^\star (\bm{x}) = v_t^{opt} (\bm{x}) \quad \forall \bm{x} \in \mathcal{Z}_t.
 \]
 Given any $v \in \mathbb{B}_b(\mathcal{Z}_{t+1})$, let
\begin{equation}\label{o_opt}
\begin{split}
(T_{t} v)(\bm{x}) &:= \inf_{\bm{u} \in \mathcal{U}(\bm{x})} \;   r(\bm{x}, \bm{u}) + 
\mathbb{E} \big [ v (f(\bm{x}, \bm{u}, \xi)) \big ]   \\
&=  \inf_{\bm{u} \in \mathcal{U}(\bm{x})} \;   r(\bm{x}, \bm{u}) + 
\sum_{s=1}^N p_s v (f(\bm{x}, \bm{u}, \hat{\xi}^{[s]}))
\end{split}
\end{equation}
for all $\bm{x} \in \mathcal{Z}_t$, where
\[
p_s := \mathrm{Prob}( \xi = \hat{\xi}^{[s]} ) \in [0,1],
\]
and thus $\sum_{s=1}^N p_s = 1$ by definition. 
Under Assumption~\ref{ass:sc}, the restricted optimal value functions satisfy the following Bellman equation:
\[
v_t^\star (\bm{x}) = (T_t v_{t+1}^\star) (\bm{x}) \quad \forall \bm{x} \in \mathcal{Z}_t.
\]
For any $\bm{x} \in \mathcal{Z}_t$, $v_t^\star (\bm{x})$ can be computed by using $v_{t+1}^\star$,
given our choice of $\mathcal{Z}_t$'s in Section~\ref{sec:dis}.
This computation does not require information about $v_{t+1}^{opt}$ outside $\mathcal{Z}_{t+1}$. 
Our goal is to develop a convex optimization-based approach to approximating the restricted optimal value function $v_t^\star$ in a convergent manner.

\section{Convex Value Functions}\label{sec:conv}

\subsection{\bf Approximation of the Bellman Operator}

When computing the optimal value functions via the DP algorithm 
\[
v_{t}^\star := T_t v_{t+1}^\star,
\]
  a closed-form expression of $v_t^\star$ is unavailable in general. Unfortunately, it is challenging to solve the optimization problem in the definition~\eqref{o_opt} of $T_t$ without a closed-form expression of $v_t^\star$.
To resolve this issue, we propose an approximation method that uses the value of $v_t^\star$ only at the nodes $\bm{x}^{[1]}, \ldots, \bm{x}^{[M_{t+1}]}$ in $\mathcal{Z}_{t+1}$.

Given any $v \in \mathbb{B}_b (\mathcal{Z}_{t+1})$, 
let 
\begin{equation}\label{con_opt}
\begin{split}
(\hat{T}_t v ) (\bm{x}) := \inf_{\bm{u}, \gamma}  \quad & r (\bm{x}, \bm{u}) +  \sum_{s=1}^N \sum_{i=1}^{M_{t+1}} p_s \gamma_{s,i} v (\bm{x}^{[i]}) \\
\mbox{s.t.} \quad & f  (\bm{x}, \bm{u}, \hat{\xi}^{[s]}) = \sum_{i=1}^{M_{t+1}} \gamma_{s,i} \bm{x}^{[i]} \quad \forall s \in \mathcal{S}\\ 
&\bm{u} \in \mathcal{U}(\bm{x}), \; \gamma_{s}  \in \Delta  \quad   \forall s \in \mathcal{S}
\end{split}
\end{equation}
for each $\bm{x} \in \mathcal{Z}_{t}$, 
where $\Delta$ is the $({M_{t+1}} - 1)$-dimensional probability simplex, i.e., $\Delta:= \{ \gamma \in \mathbb{R}^{M_{t+1}} : \sum_{i=1}^{M_{t+1}} \gamma_i = 1, \gamma_i \geq 0 \; \forall i=1, \ldots, {M_{t+1}}\}$, and $\mathcal{S}:= \{1, \ldots, N\}$.
Under Assumption~\ref{ass:sc}, the objective function of \eqref{con_opt} is lower semicontinuous in $(\bm{u}, \gamma) \in \mathcal{U}(\bm{x}) \times \Delta^N$, and the feasible set is compact. Therefore, by \cite[Proposition D.5]{Hernandez2012}, \eqref{con_opt} admits an optimal solution, and $\hat{T}_t$ maps $\mathbb{B}_b(\mathcal{Z}_{t+1})$ to $\mathbb{B}_b (\mathcal{Z}_t)$.
Moreover, $\hat{T}_t$ is monotone, i.e., for any $v, v' \in \mathbb{B}_b(\mathcal{Z}_{t+1})$ such that $v \leq v'$, we have $\hat{T}_t v \leq \hat{T}_t v'$.
The constrained-optimization problem in \eqref{con_opt} for a fixed $\bm{x} \in \mathcal{Z}_t$ can be numerically solved to obtain a globally optimal solution  using several existing convex optimization algorithms (e.g.,~\cite{Nesterov2018, Boyd2004, Bertsekas2015}) if the problem is convex,  i.e., $\bm{u} \mapsto r(\bm{x}, \bm{u})$ is a convex function and $\bm{u} \mapsto f(\bm{x}, \bm{u}, \hat{\xi}^{[s]})$ is an affine function.\footnote{More precisely, the set $\mathcal{U}(\bm{x})$ needs to be represented by convex inequalities, i.e.,
there exist functions $a_k: \mathcal{X} \times \mathbb{R}^m \to \mathbb{R}$ and $b_k: \mathcal{X} \to \mathbb{R}$ such that
\[
\mathcal{U}(\bm{x}) := \{ \bm{u} \in \mathbb{R}^m : a_k (\bm{x}, \bm{u}) \leq b_k(\bm{x}), k=1, \ldots, N_{ineq}\},
\]
 where $\bm{u} \mapsto a_k (\bm{x}, \bm{u})$ is a convex function for each fixed $\bm{x} \in \mathcal{X}$ and each $k$.}

 Note that for each sample index $s$ the term $v(f(\bm{x}, \bm{u}, \hat{\xi}^{[s]}))$ in the original Bellman operator is approximated by $\sum_{i=1}^{M_{t+1}} \gamma_{s,i} v (\bm{x}^{[i]})$, where the auxiliary weight variable $\gamma_{s,i}$ can be interpreted as the degree to which $f (\bm{x}, \bm{u}, \hat{\xi}^{[s]})$ is represented by $\bm{x}^{[i]}$.
This idea of representing the next state as a convex combination of grid points or nodes has also been adopted in the analysis and design of semi-Lagrangian schemes for Hamilton--Jacobi--Bellman partial differential equations arising in continuous-time optimal control problems~\cite{Falcone2013}.  
Although this paper focuses on discrete-time DP problems, it would be an interesting future research to extend our method to the continuous-time setting, possibly using advanced techniques developed in the recent literature on semi-Lagrangian methods~(e.g., \cite{Alla2019, Picarelli2020}). 

When $v$ is convex, we immediately observe that $T_t v$ is upper bounded by $\hat{T}_t v$  on $\mathcal{Z}_{t}$  because
\[
v (f(\bm{x}, \bm{u}, \hat{\xi}^{[s]})) = v \bigg (\sum_{i=1}^{M_{t+1}} \gamma_{s,i} \bm{x}^{[i]}  \bigg )  \leq \sum_{i=1}^{M_{t+1}}  \gamma_{s,i} v (\bm{x}^{[i]}).
\]
Further showing that $T_t v$ is lower bounded by $\hat{T}_t v - C\delta$ for some positive constant $C$, we prove that $\hat{T}_t v$ converges uniformly to ${T}_t v$ on $\mathcal{Z}_{t}$ as the maximum distance $\delta$ between neighboring nodes tends to zero.
By definition, $\hat{T}_t$ depends  on the discretization parameter $\delta$ as well as $N$. For notational simplicity, however, we suppress the dependence on $\delta$ and $N$.

\subsection{\bf Error Bound and Convergence}

In this section, we assume the convexity of $v$ and  show the uniform convergence property of our method. This assumption will be relaxed in Section~\ref{sec:nonconv}.
\begin{proposition}\label{prop:bd}
Suppose that Assumption~\ref{ass:sc} holds, and that 
$v \in \mathbb{L}_b(\mathcal{Z}_{t+1})$ is convex. 
Then, we have
\[
(\hat{T}_t v )(\bm{x}) -  L_v \delta \leq (T_t v )(\bm{x}) \leq (\hat{T}_t v ) (\bm{x}) \quad \forall \bm{x} \in \mathcal{Z}_{t},
\]
where
$L_v := \max_{j = 1,\ldots, N_{\mathcal{C}, t+1}} \sup_{\bm{x}, \bm{x}' \in \mathcal{C}_j: \bm{x} \neq \bm{x}'} \frac{\| v (\bm{x}) - v (\bm{x}')  \|}{\| \bm{x} - \bm{x}'  \|}$, and therefore $\hat{T}_t v$ converges uniformly to $T_t v$ on $\mathcal{Z}_t$ as $\delta \to 0$.
\end{proposition}
 Note that $L_v < \infty$ because $v$ is continuous on the compact set $\mathcal{Z}_{t+1}$ under the assumption of convexity and lower semicontinuity.
\begin{proof}
Fix an arbitrary $\bm{x} \in \mathcal{Z}_t$. 
We first show that $(T_t v) (\bm{x}) \leq (\hat{T}_t v)(\bm{x})$.
Under Assumption~\ref{ass:sc}, the objective function of \eqref{con_opt} is lower semicontinuous, and the feasible set is compact. 
Let $(\hat{\bm{u}}, \hat{\gamma}) \in \mathcal{U}(\bm{x}) \times \Delta^N$ be an optimal solution to \eqref{con_opt}, i.e., it satisfies
\begin{equation}\nonumber
\begin{split}
&(\hat{T}_t v)(\bm{x}) = r(\bm{x}, \hat{\bm{u}}) +  \sum_{s=1}^N \sum_{i=1}^{M_{t+1}} p_s \hat{\gamma}_{s,i} v  (\bm{x}^{[i]})\\
&f(\bm{x}, \hat{\bm{u}}, \hat{\xi}^{[s]}) = \sum_{i=1}^{M_{t+1}} \hat{\gamma}_{s,i} \bm{x}^{[i]} \in {\mathcal{Z}_{t+1}} \quad \forall s \in \mathcal{S}.
\end{split}
\end{equation}
The convexity of $v$ on $\mathcal{Z}_{t+1}$  implies that
\begin{equation} \nonumber
\begin{split}
 v \bigg (
 \sum_{i=1}^{M_{t+1}} \hat{\gamma}_{s,i} \bm{x}^{[i]}
 \bigg ) \leq \sum_{i=1}^{M_{t+1}} \hat{\gamma}_{s,i} v (\bm{x}^{[i]}) \quad \forall s \in \mathcal{S}.
\end{split}
\end{equation} 
Therefore, by the definition of $(\hat{\bm{u}}, \hat{\gamma})$, we have
\begin{equation} \nonumber
\begin{split}
(\hat{T}_t v)(\bm{x}) &\geq r(\bm{x}, \hat{\bm{u}}) +  \sum_{s=1}^N p_s v \bigg (
 \sum_{i=1}^{M_{t+1}} \hat{\gamma}_{s,i} \bm{x}^{[i]}
 \bigg )\\
 &= r(\bm{x}, \hat{\bm{u}}) +  \sum_{s=1}^N p_s v (f(\bm{x}, \hat{\bm{u}},
 \hat{\xi}^{[s]}))\\
 &\geq \inf_{\bm{u} \in \mathcal{U}(\bm{x})} \;   r(\bm{x}, {\bm{u}}) + \sum_{s=1}^N p_s v  (f(\bm{x}, {\bm{u}},
  \hat{\xi}^{[s]})) \\
 &= (T_t v )(\bm{x}).
\end{split}
\end{equation}

We now show that $(\hat{T}_t v)(\bm{x})- L_v \delta \leq (T_t v)(\bm{x})$.
Under Assumption~\ref{ass:sc}, 
the objective function of \eqref{o_opt} is lower semicontinuous~\cite[Lemma 8.5.5]{Hernandez2012b}, and the feasible set is compact. 
Thus, by \cite[Proposition D.5]{Hernandez2012}, 
 it admits an optimal solution, i.e.,
there exists $\bm{u}^\star \in \mathcal{U}(\bm{x})$ such that
\[
(T_t v )(\bm{x}) =  r(\bm{x}, \bm{u}^\star) + \sum_{s=1}^N p_s v  (f (\bm{x}, \bm{u}^\star, \hat{\xi}^{[s]})).
\]
Let
\[
y_s:= f(\bm{x}, \bm{u}^\star, \hat{\xi}^{[s]}) \quad \forall s \in \mathcal{S}.
\]
Then, $y_s  \in \mathcal{Z}_{t+1}$ because $\bm{x} \in \mathcal{Z}_t$.
Thus, 
 there exists a unique $j \in \{1, \ldots, N_{\mathcal{C}, t+1}\}$ such that
$y_s \in \mathcal{C}_j$.
Let $\mathcal{N}(y_s)$ denote the set of grid points on the cell $\mathcal{C}_j$ that contains $y_s$.
For each $s$, choose $\gamma_{s}^\star \in \Delta$ such that
$f(\bm{x}, \bm{u}^\star, \hat{\xi}^{[s]}) = \sum_{i \in \mathcal{N}(y_s)} \gamma_{s,i}^\star \bm{x}^{[i]}$ and $\sum_{i \in \mathcal{N}(y_s)} \gamma_{s,i}^\star = 1$. Note that $\gamma_{s,i}^\star = 0$ for all $i \notin \mathcal{N}(y_s)$.
We then have
\begin{equation}\label{tv}
\begin{split}
(T_t  v )(\bm{x}) 
 &= r(\bm{x}, \bm{u}^\star) +  \sum_{s=1}^N p_s v (f( \bm{x}, \bm{u}^\star, \hat{\xi}^{[s]}))\\
&= r(\bm{x}, \bm{u}^\star) +  \sum_{s=1}^N p_s v\bigg (
 \sum_{i \in \mathcal{N} (y_s)} \gamma_{s,i}^\star \bm{x}^{[i]}
 \bigg ).
\end{split}
\end{equation}
By the definition of $L_v$, we have
\[
v  \bigg (
 \sum_{i \in \mathcal{N} (y_s)} \gamma_{s,i}^\star \bm{x}^{[i]}
 \bigg ) \geq 
v (\bm{x}^{[i']}) - L_v  \bigg \| \bm{x}^{[i']} 
-\sum_{i \in \mathcal{N} (y_s)} \gamma_{s,i}^\star \bm{x}^{[i]}
\bigg  \|
\]
  for all $i' \in \mathcal{N}(y_s)$.
This implies that
\begin{equation} \label{ineq}
\begin{split}
v \bigg (
 \sum_{i \in \mathcal{N} (y_s)} \gamma_{s,i}^\star \bm{x}^{[i]}
 \bigg ) 
&\geq \max_{i' \in \mathcal{N}(y_s)} \bigg [ v (\bm{x}^{[i']}) - L_v  \bigg \| \bm{x}^{[i']} 
-\sum_{i \in \mathcal{N} (y_s)} \gamma_{s,i}^\star \bm{x}^{[i]}
 \bigg \| \bigg ]\\
 &\geq  \max_{i' \in \mathcal{N}(y_s)} \bigg [ v (\bm{x}^{[i']}) - L_v   \sum_{i \in \mathcal{N} (y_s)} \gamma_{s,i}^\star \| \bm{x}^{[i']} 
- \bm{x}^{[i]}
 \| \bigg ]\\
&\geq \max_{i' \in \mathcal{N}(y_s)} \bigg [ v (\bm{x}^{[i']})- L_v  \sum_{i \in \mathcal{N} (y_s)} \gamma_{s,i}^\star \delta \bigg ]\\
&\geq  \sum_{i' \in \mathcal{N}(y_s)} \gamma_{s, i'}^\star ( v (\bm{x}^{[i']}) - L_v  \delta )\\
&= \sum_{i' = 1}^{M_{t+1}} \gamma_{s, i'}^\star v (\bm{x}^{[i']}) - L_v   \delta,
\end{split}
\end{equation}
where the third inequality holds by the definition of $\delta$, and 
the last equality  holds because $\sum_{i' \in \mathcal{N}(y_s)} \gamma_{s,i'}^\star = 1$.
Combining \eqref{tv} and \eqref{ineq}, we obtain that
\begin{equation}\nonumber
\begin{split}
(T_t v )(\bm{x}) 
 &\geq r(\bm{x}, \bm{u}^\star) +   \sum_{s=1}^N \sum_{i=1}^{M_{t+1}} p_s \gamma_{s,i}^\star v (\bm{x}^{[i]}) -  L_v \delta\\
&\geq (\hat{T}_t v )(\bm{x}) -  L_v \delta,
\end{split}
\end{equation}
where the second inequality holds because $(\bm{u}^\star, \gamma^\star)$ is a feasible solution to \eqref{con_opt}.
Therefore, the result follows. 
\end{proof}

Let $\hat{v}_t: \mathcal{Z}_t \to \mathbb{R}$ be defined by 
\[
\hat{v}_t (\bm{x}) :=( \hat{T}_t \hat{v}_{t+1})(\bm{x}) \quad \forall \bm{x} \in \mathcal{Z}_t
\]
for $t = 0, \ldots, K-1$,
with $\hat{v}_K \equiv q$ on $\mathcal{Z}_K$.
We now show that $\hat{v}_t$  converges uniformly to the optimal value function $v_t^\star$ on $\mathcal{Z}_t$ as $\delta$ tends to zero when all $v_t^\star$'s are convex.  
A sufficient condition for the convexity of $v_t^\star$'s is provided in Section~\ref{sec:vi}.

\begin{theorem}[Uniform Convergence and Error Bound I]\label{thm:bd1}
Suppose that Assumption~\ref{ass:sc} holds, and that the optimal value function $v_t^\star$ is convex on $\mathcal{Z}_{t}$ for all $t = 0, \ldots, K$.
Then, we have
\begin{equation}\label{fin_bd}
  0 \leq \hat{v}_t(\bm{x}) - v_t^\star (\bm{x}) \leq \sum_{k = t+1}^K L_k \delta   \quad \forall \bm{x} \in \mathcal{Z}_t \; \forall t =0, \ldots, K,
\end{equation}
where $L_k :=  \sup_{j = 1,\ldots, N_{\mathcal{C},k}} \sup_{\bm{x}, \bm{x}' \in \mathcal{C}_j: \bm{x} \neq \bm{x}'} \frac{\| v_k^\star(\bm{x}) - v_k^\star(\bm{x}')  \|}{\| \bm{x} - \bm{x}'  \|}$,
and therefore
$\hat{v}_t$ converges uniformly to $v_t^\star$ on $\mathcal{Z}_t$ as $\delta \to 0$.
\end{theorem}

\begin{proof}
We use mathematical induction to prove \eqref{fin_bd}.
For $t = K$,  $\hat{v}_K = v_K^\star \equiv q$ on $\mathcal{Z}_K$, and thus \eqref{fin_bd} holds.
Suppose that the induction hypothesis holds for some $t$.
By applying the operator $\hat{T}_{t-1}$ on all sides of \eqref{fin_bd}, 
we have for all $\bm{x} \in \mathcal{Z}_{t-1}$
\[
(\hat{T}_{t-1} \hat{v}_t)(\bm{x}) - \sum_{k=t+1}^K L_k \delta \leq (\hat{T}_{t-1}v_t^\star)(\bm{x}) \leq (\hat{T}_{t-1} \hat{v}_t)(\bm{x})
\]
by the monotonicity of $\hat{T}_{t-1}$.
On the other hand,
by Proposition~\ref{prop:bd}, 
\[
(\hat{T}_{t-1} v_{t}^\star)(\bm{x})  -  L_{t} \delta \leq (T_{t-1} v_{t}^\star)(\bm{x}) \leq (\hat{T}_{t-1}v_{t}^\star)(\bm{x}) \quad \forall \bm{x} \in \mathcal{Z}_{t-1}
\]
because $v_t^\star \in \mathbb{L}_b(\mathcal{Z}_t)$~\cite[Lemma 8.5.5]{Hernandez2012b} and it is convex.
Therefore, we conclude that
\[
(\hat{T}_{t-1} \hat{v}_t) (\bm{x}) - \sum_{k = t}^K L_k \delta \leq (T_{t-1} v_t^\star)(\bm{x}) \leq (\hat{T}_{t-1} \hat{v}_t)(\bm{x}) \; \forall \bm{x} \in \mathcal{Z}_{t-1}.
\]
Note that $\hat{v}_{t-1} = \hat{T}_{t-1} v_t$  and $v_{t-1}^\star = T_{t-1} v_t^\star$ on $\mathcal{Z}_{t-1}$ by definition.
This implies that $\hat{v}_{t-1} (\bm{x}) - \sum_{k=t}^K L_k \delta \leq v_{t-1}^\star (\bm{x}) \leq \hat{v}_{t-1} (\bm{x})$ for all $\bm{x} \in \mathcal{Z}_{t-1}$ as desired. \qed
\end{proof}

The approximate value function $\hat{v}_{t+1}$ evaluated at the nodes $\bm{x}^{[1]}, \ldots, \bm{x}^{[M_{t+1}]}$ for $t = 0, \ldots, K$ can be used to construct a deterministic stationary policy, $\hat{\pi} := (\hat{\pi}_0, \ldots, \hat{\pi}_{K-1})$, by setting 
\begin{equation}\nonumber
\begin{split}
\hat{\pi}_t (\bm{x}) \in \argmin_{\bm{u} \in \mathcal{U}(\bm{x})} \min_{\gamma \in \Delta^N} &\bigg [r(\bm{x}, \bm{u}) +  \sum_{s=1}^N \sum_{i=1}^{M_{t+1}} p_s \gamma_{s,i} \hat{v}_{t+1}(\bm{x}^{[i]}) \bigg ]\\
\mbox{s.t.} \; & f(\bm{x}, \bm{u}, \hat{\xi}^{[s]}) = \sum_{i=1}^{M_{t+1}} \gamma_{s,i} \bm{x}^{[i]} \quad \forall s \in \mathcal{S}
\end{split}
\end{equation}
for all $\bm{x} \in \mathcal{Z}_{t}$.
Let $v_t^{\hat{\pi}}: \mathcal{Z}_t \to \mathbb{R}$ be defined by
\begin{equation} \nonumber
\begin{split}
v_t^{\hat{\pi}} (\bm{x}) &:= (T^{\hat{\pi}_t}v_{t+1}^{\hat{\pi}}) (\bm{x})\\
&:= r(\bm{x}, \hat{\pi}_t(\bm{x})) +  \sum_{s=1}^N  p_s v_{t+1}^{\hat{\pi}}(f(\bm{x}, \hat{\pi}_t(\bm{x}), \hat{\xi}^{[s]})) \quad \forall \bm{x} \in \mathcal{Z}_t
\end{split}
\end{equation}
with $v_K^{\hat{\pi}} \equiv q$ on $\mathcal{Z}_K$. 
It is straightforward to check that $T^{\hat{\pi}_t}$ is monotone. 
To show that the cost-to-go function $v_t^{\hat{\pi}}$ of  $\hat{\pi}$ converges uniformly to the optimal value function $v_t^\star$, we first observe the following:

\begin{lemma}\label{lem:compare}
Suppose that Assumption~\ref{ass:sc}  holds, and that $\hat{v}_t$ is convex on $\mathcal{Z}_t$ for all $t=0, \ldots, K$. Then, we have
\[
v_t^{\hat{\pi}} (\bm{x}) \leq \hat{v}_t (\bm{x}) \quad \forall \bm{x} \in \mathcal{Z}_t
\]
for all $t=0, \ldots, K$.
\end{lemma}

\begin{proof}
We use mathematical induction. 
For $t = K$, $v_K^{\hat{\pi}} = \hat{v}_K \equiv q$ on $\mathcal{Z}_K$, and thus the induction hypothesis holds. 
Suppose that $v_{t+1}^{\hat{\pi}}  \leq \hat{v}_{t+1}$  on $\mathcal{Z}_{t+1}$ for some ${t+1}$.
By the monotonicity of $T^{\hat{\pi}_t}$, we have 
\begin{equation} \label{mono}
(T^{\hat{\pi}_t} v_{t+1}^{\hat{\pi}})(\bm{x}) \leq (T^{\hat{\pi}_t} \hat{v}_{t+1})(\bm{x}) \quad \forall \bm{x} \in \mathcal{Z}_{t}.
\end{equation}
Fix an arbitrary $\bm{x} \in \mathcal{Z}_{t}$.
By the definition of $\hat{T}_t$ and $\hat{\pi}$,
under Assumption~\ref{ass:sc}, there exists $\hat{\gamma}\in \Delta^N$ such that
\begin{equation}\label{compare1}
(\hat{T}_t \hat{v}_{t+1}) (\bm{x})  = r(\bm{x}, \hat{\pi}_t(\bm{x})) 
+   \sum_{s=1}^N \sum_{i=1}^{M_{t+1}} p_s \hat{\gamma}_{s,i} \hat{v}_{t+1}(\bm{x}^{[i]})
\end{equation}
and
\[
f(\bm{x}, \hat{\pi}_t(\bm{x}), \hat{\xi}^{[s]}) = \sum_{i=1}^{M_{t+1}} \hat{\gamma}_{s,i} \bm{x}^{[i]} \quad \forall s \in \mathcal{S}.
\]
By the convexity of $\hat{v}_{t+1}$, we have
\begin{equation}\label{compare2}
\begin{split}
\sum_{i=1}^{M_{t+1}} \hat{\gamma}_{s,i} \hat{v}_{t+1} (\bm{x}^{[i]}) &\geq
\hat{v}_{t+1} \bigg ( \sum_{i=1}^{M_{t+1}} \hat{\gamma}_{s,i} \bm{x}^{[i]} \bigg )\\
&= \hat{v}_{t+1} (f(\bm{x}, \hat{\pi}_t(\bm{x}), \hat{\xi}^{[s]}) )
\end{split}
\end{equation}
for each $s \in \mathcal{S}$.
By combining the inequalities \eqref{compare1} and \eqref{compare2}, we obtain that
\begin{equation}\nonumber
\begin{split}
(\hat{T}_t \hat{v}_{t+1}) (\bm{x})   &\geq
r(\bm{x}, \hat{\pi}_t(\bm{x})) + \sum_{s=1}^N p_s \hat{v}_{t+1} (f(\bm{x}, \hat{\pi}_t(\bm{x}), \hat{\xi}^{[s]}) )\\
&= (T^{\hat{\pi}_t} \hat{v}_{t+1} )(\bm{x}).
\end{split}
\end{equation}
Recall \eqref{mono}, we conclude that for all $\bm{x} \in \mathcal{Z}_t$,
\[
v_{t}^{\hat{\pi}} (\bm{x})  = (T^{\hat{\pi}_t} v_{t+1}^{\hat{\pi}})(\bm{x}) \leq (T^{\hat{\pi}_t} \hat{v}_{t+1})(\bm{x})
\leq (\hat{T}_t \hat{v}_{t+1})(\bm{x}) = \hat{v}_{t}(\bm{x}).
\]
This completes mathematical induction, and the result follows. \qed
\end{proof}

By using this lemma and Theorem~\ref{thm:bd1}, we obtain the following uniform convergence result:
\begin{theorem}[Uniform Convergence and Error Bound II]\label{cor:bd_conv}
Suppose that Assumption~\ref{ass:sc}  holds, and that $v^\star_t$ and $\hat{v}_t$ are convex on $\mathcal{Z}_t$ for all $t$.
Then, we have
\[
0 \leq  {v}_t^{\hat{\pi}}(\bm{x}) - v_t^\star (\bm{x}) \leq \sum_{k = t+1}^K L_k \delta \quad \forall \bm{x} \in \mathcal{Z}_t \; \forall t =0, \ldots, K,
\]
where $L_k :=  \sup_{j = 1,\ldots, N_{\mathcal{C},k}} \sup_{\bm{x}, \bm{x}' \in \mathcal{C}_j: \bm{x} \neq \bm{x}'} \frac{\| v_k^\star(\bm{x}) - v_k^\star(\bm{x}')  \|}{\| \bm{x} - \bm{x}'  \|}$, and therefore
${v}_t^{\hat{\pi}}$ converges uniformly to $v_t^\star$ on $\mathcal{Z}_t$ as $\delta \to 0$.
\end{theorem}

\begin{proof}
Fix an arbitrary $t \in \{0, \ldots, K\}$.
By Theorem~\ref{thm:bd1} and Lemma~\ref{lem:compare},
we first observe that 
\[
v_t^{\hat{\pi}}(\bm{x}) - v_t^\star(\bm{x}) \leq \hat{v}_t (\bm{x})- v_t^\star(\bm{x}) \leq \sum_{k=t+1}^K L_k \delta \quad \forall \bm{x} \in \mathcal{Z}_t.
\]
In addition, we have 
$v_t^\star \leq v_t^{\hat{\pi}}$ since $v_t^\star$ is the minimal cost-to-go function under Assumption~\ref{ass:sc}.
Therefore, the result follows. \qed
\end{proof}

\subsection{\bf Interpolation-Free DP Algorithm} \label{sec:vi}

The error bounds and convergence results established in the previous subsection are valid if $v_t^\star$  and $\hat{v}_t$ are convex.
We now introduce a sufficient condition for the convexity $v_t^\star$ and $\hat{v}_t$.

\begin{assumption}[Linear-Convex Control]\label{ass:lc}
\begin{enumerate}
\item
The function
\[
(\bm{x}, \bm{u}) \mapsto  f(\bm{x}, \bm{u}, \bm{\xi})
\]
is affine on $\mathcal{K} := \{ (\bm{x}, \bm{u}) : \bm{x} \in \mathcal{X}, \bm{u} \in \mathcal{U}(\bm{x}) \}$ for each $\bm{\xi} \in \Xi$. In addition,  the stage-wise cost function $r$ is convex on $\mathcal{K}$, and the terminal cost function $q$ is convex on $\mathcal{X}$.

\item 
If $\bm{u}_{(k)} \in \mathcal{U}(\bm{x}_{(k)})$ for $k=1,2$, then 
\[
\lambda \bm{u}_{(1)} + (1-\lambda) \bm{u}_{(2)} \in \mathcal{U}(\lambda \bm{x}_{(1)} + (1-\lambda) \bm{x}_{(2)})
\]
 for any $\bm{x}_{(1)}, \bm{x}_{(2)} \in \mathcal{X}$ and any $\lambda \in (0,1)$.

\end{enumerate}
\end{assumption}

\begin{lemma}\label{lem:conv}
Suppose that Assumptions~\ref{ass:sc} and~\ref{ass:lc} hold, and that $v:\mathcal{Z}_{t+1} \to \mathbb{R}$ is convex. Then, ${T}_t v, \hat{T}_t v: \mathcal{Z}_t \to \mathbb{R}$ are convex.
\end{lemma}

Its proof can be found in~Appendix~\ref{app:conv}.
An immediate observation obtained from Lemma~\ref{lem:conv} is the convexity of~$v_t^\star$ and $\hat{v}_t$ for all $t = 0, \ldots, K$ because $v_K^\star$ and $\hat{v}_K$ are convex on $\mathcal{Z}_K$.

\begin{proposition}\label{prop:conv}
Under Assumptions~\ref{ass:sc} and \ref{ass:lc},
the functions $v_t^\star$ and $\hat{v}_t$ are convex for all $t =0, \ldots, K$. 
\end{proposition}

By Proposition~\ref{prop:conv}, the error bounds and convergence results in Theorems~\ref{thm:bd1} and \ref{cor:bd_conv}   are valid under Assumptions \ref{ass:sc}--\ref{ass:lc}.
Note, however, that Assumption~\ref{ass:lc} is merely a \emph{sufficient} condition for convexity.

\begin{algorithm}[tp]
Initialize Initialize $\hat{v}_K \equiv q$ on $\mathcal{Z}_K$;\\
\For{$t = K-1: -1: 0$}{
Set $\hat{v}_{t}(\bm{x}^{[i]}) := (\hat{T}_t \hat{v}_{t+1})(\bm{x}^{[i]})$  by solving the problem in \eqref{con_opt} with $v \equiv \hat{v}_{t+1}$ for each $i=1, \ldots, M_{t}$;\\
Set $\hat{\pi}_t (\bm{x}^{[i]})$ as an optimal $\bm{u}$ of the problem in \eqref{con_opt} with $v \equiv \hat{v}_{t+1}$ for each $i=1, \ldots, M_{t}$;
}
\caption{Interpolation-Free DP algorithm
}
\label{algorithm:dp}
\end{algorithm}

Algorithm~\ref{algorithm:dp}, which is based on the proposed method, can be used to evaluate the approximate value function $\hat{v}_t$ and to construct the corresponding policy $\hat{\pi}_t$ at all grid point $\bm{x}^{[i]}$'s in $\mathcal{Z}_t$.  
Given $\hat{v}_t (\bm{x}^{[i]})$ for all $i=1, \ldots, M_t$, 
the value of $\hat{v}_t$ at the other points in $\mathcal{Z}_t$, the approximate value function can be computed using \eqref{con_opt} with $v := \hat{v}_{t+1}$.
We also observe that the optimization problem in \eqref{con_opt} used to evaluate $(\hat{T}_t v)(\bm{x})$ is convex regardless of the convexity of $v$ under Assumption~\ref{ass:lc}.
Thus, in each iteration of the proposed DP algorithm, it suffices to solve $M_{t}$ convex optimization problems, each of which is for $\bm{x} := \bm{x}^{[i]}$.

\begin{proposition}\label{prop:cp}
Under Assumption~\ref{ass:lc}, the optimization problem in \eqref{con_opt} for any fixed $\bm{x} \in \mathcal{Z}_{t}$ is a convex optimization problem if $v\in \mathbb{B}_b(\mathcal{Z}_{t+1})$.
\end{proposition}
\begin{proof}
Fix an arbitrary state $\bm{x} \in \mathcal{Z}_t$.
The objective function is convex in $\bm{u}$ and linear in $\gamma$ (even when $v$ is nonconvex).
Furthermore, the equality constraints are linear in $(\bm{u}, \gamma)$.
Assumption~\ref{ass:lc} implies that $\mathcal{U}(\bm{x})$ is convex. 
Also, the probability simplex $\Delta$ is convex.
Therefore, this optimization problem for any fixed $\bm{x}$ is convex. \qed
\end{proof}

\begin{remark}[Interpolation-free property]\label{rem:if}
Note that we can evaluate $\hat{v}_t$ at an arbitrary $\bm{x} \in \mathcal{Z}_t$ by using the definition~\eqref{con_opt} of $\hat{T}_t$ without any explicit interpolation that may introduce additional numerical errors. 
This feature is also useful  when the output of $\hat{\pi}_t$ needs to be specified at a particular state that is different from the grid points $\{\bm{x}^{[i]}\}_{i=1}^{M_t}$.
Unlike many existing discretization-based methods, 
our approach does not require a separate interpolation stage in constructing both the optimal value function and control policies. 
\end{remark}

The computational complexity of Algorithm~\ref{algorithm:dp} increases exponentially with the state space dimension. This complexity issue, called the \emph{curse of dimensionality}, is inherited from dynamic programming. 
Several approximate dynamic programming (ADP) and RL methods have been proposed to alleviate the dimensionality issue using approximation in value spaces~(e.g., \cite{Mnih2015}), policy spaces~(e.g., \cite{Schulman2015}) or both~(e.g., \cite{Lilicrap2015, Haarnoja2018}).
In ADP and RL methods, it is important to find a reasonable balance between scalable implementation and adequate performance.
Our focus is to pursue the latter with a uniform convergence guarantee.
Nevertheless, an important future research is to improve the scalability of  the proposed convex optimization approach by replacing a grid with a function approximator.

\section{Nonconvex Value Functions}\label{sec:nonconv}

The convexity of the optimal value function plays a critical role in obtaining the convergence results in the previous section. 
To relax the convexity condition (e.g., Assumption~\ref{ass:lc}), we  first show that the proposed approximation method is useful when constructing a suboptimal policy with a provable error bound in the case of nonlinear control-affine systems with convex cost functions.
For further general cases, we propose a modified approximation method  based on a nonconvex bi-level optimization problem, where the inner problem is a linear program.

\subsection{\bf Control-Affine Systems}\label{sec:ca}

Consider a control-affine system of the form
\begin{equation}\label{sys_ca}
x_{t+1} = f(x_t, u_t, \xi_t) := g(x_t, \xi_t) + h(x_t, \xi_t)u_t.
\end{equation}
More precisely, we assume the following:
\begin{assumption}\label{ass:ca}
The function $f:\mathcal{X} \times \mathcal{U} \times \Xi \to \mathcal{X}$ can be expressed as \eqref{sys_ca}, where 
 $g:\mathcal{X}\times \Xi \to \mathbb{R}^n$ and $h: \mathcal{X} \times \Xi \to \mathbb{R}^{n\times m}$ are (possibly nonlinear) measurable functions such that $g(\cdot, \bm{\xi})$ and $h(\cdot, \bm{\xi})$ are continuous for each\\ $\bm{\xi} \in \Xi$.
In addition, $\bm{u} \mapsto r(\bm{x}, \bm{u})$ is convex on $\mathcal{U}(\bm{x})$ for each $\bm{x} \in \mathcal{X}$. 
\end{assumption}

Note that the condition on $r$ imposed by this assumption is weaker than Assumption \ref{ass:lc} which requires the joint convexity of $r$.
In this setting, each iteration of the DP algorithm in Section~\ref{sec:vi} still involves $M_{t}$ convex optimization problems. 

\begin{proposition} \label{prop:ca_conv}
Under Assumption~\ref{ass:ca}, the optimization problem in \eqref{con_opt} is a convex program.
\end{proposition}

Due to the nonconvexity of the optimal value function, $\hat{v}_t$ obtained by
value iteration in the previous section  is no longer guaranteed to converge to the optimal value function as $\delta$ tends to zero. 
However, we are still able to characterize an error bound for the approximate policy $\hat{\pi}$ as follows:

\begin{proposition}[Error Bound]\label{thm:bound}
Suppose that Assumptions~\ref{ass:sc} and \ref{ass:ca} hold, and that $v^\star_t$ is  locally Lipschitz continuous on $\mathcal{Z}_t$ for each $t$. Then,  for all $t =0, \ldots, K$, we have
\begin{equation}\nonumber
\begin{split}
0 &\leq {v}^{\hat{\pi}}_t (\bm{x}) - v^\star_t (\bm{x}) \leq  v_t^{\hat{\pi}} (\bm{x}) - \hat{v}_t(\bm{x}) + \sum_{k=t+1}^K L_k \delta \quad \forall \bm{x} \in \mathcal{Z}_t,
\end{split}
\end{equation}
where
$L_k :=  \sup_{j = 1,\ldots, N_{\mathcal{C},k}} \sup_{\bm{x}, \bm{x}' \in \mathcal{C}_j: \bm{x} \neq \bm{x}'} \frac{\| v_k^\star(\bm{x}) - v_k^\star(\bm{x}')  \|}{\| \bm{x} - \bm{x}'  \|}$.
\end{proposition}
\begin{proof}
Fix an arbitrary $t \in \{0, \ldots, K\}$.
By the optimality of $v_t^\star$, we have
\[
v_t^\star(\bm{x}) \leq v_t^{\hat{\pi}} (\bm{x}) \quad \forall \bm{x} \in \mathcal{Z}_t .
\]
Note that in the second part of the proof of Proposition~\ref{prop:bd}, the convexity of $v$ is unused. It only requires the continuity of $v$ on $\mathcal{Z}_t$. 
Therefore, the second inequality of \eqref{fin_bd} in Theorem \ref{thm:bd1} holds when $v_t^\star$ is continuous, i.e.,
\[
\hat{v}_t(\bm{x}) - v_t^\star (\bm{x}) \leq \sum_{k = t+1}^K L_k \delta   \quad \forall \bm{x} \in \mathcal{Z}_t.
\]
This implies that
\[
v_t^{\hat{\pi}}(\bm{x}) - v_t^\star(\bm{x}) \leq v_t^{\hat{\pi}}(\bm{x}) - \hat{v}_t(\bm{x}) + \sum_{k=t+1}^K L_k \delta \quad \forall \bm{x} \in \mathcal{Z}_t,
\]
and the result follows. \qed
\end{proof}

This proposition implies that the performance of $\hat{\pi}$ converges to the optimum as $\delta$ tends to zero  if its cost-to-go function $v^{\hat{\pi}}_t$ converges to the approximate value function $\hat{v}_t$ for all $t$. 
Otherwise, 
it is possible that $\hat{v}_t^\pi$ converges to some function which is different from the optimal value function. Or it may oscillate within the error bound. 
However, this {\it a posteriori} error bound is useful when we need to design a controller with a provable performance guarantee, for example, in safety-critical systems where the objective is to maximize the probability of safety (e.g.,~\cite{Abate2008}), as this  problem is subject to nonconvexity issues~\cite{Yang2018}.

\subsection{\bf General Case}\label{sec:gen}

In the case of general nonlinear systems with nonlinear stage-wise cost functions, 
we modify the operator $\hat{T}_t$ as follows.
For any $v \in \mathbb{B}_b(\mathcal{Z}_{t+1})$, let
\begin{equation}\label{nopt}
\begin{split}
(\tilde{T}_t v ) (\bm{x}) := \inf_{\bm{u} \in \mathcal{U}(\bm{x})}  \; &\big [ r(\bm{x}, \bm{u}) + (H_t  v)(\bm{x}, \bm{u}) \big ] \;\; \forall \bm{x} \in \mathcal{Z}_t, 
\end{split}
\end{equation}
where 
\begin{equation}\label{copt}
\begin{split}
(H_t v ) (\bm{x}, \bm{u}) :=
\inf_{\gamma}  \;\; &  \sum_{s=1}^N \sum_{i \in \mathcal{N} (y_s)} p_s \gamma_{s,i} v (\bm{x}^{[i]}) \\
\mbox{s.t.} \;\; &  y_s = \sum_{i \in \mathcal{N} (y_s)} \gamma_{s, i} \bm{x}^{[i]}  \quad \forall s  \in \mathcal{S}\\ 
&\gamma_{s}  \in \Delta  \quad   \forall s \in \mathcal{S}\\
&\gamma_{s,i} = 0 \quad \forall i \notin \mathcal{N} (y_s)\quad \forall s \in \mathcal{S}
\end{split}
\end{equation}
for each $(\bm{x}, \bm{u}) \in \mathcal{Z}_t \times \mathcal{U}$ such that $\bm{u} \in \mathcal{U}(\bm{x})$, and $y_s=f(\bm{x}, \bm{u}, \hat{\xi}^{[s]})$ for all $s  \in \mathcal{S}$. Here
 $\mathcal{N} (y_s)$ denotes the set of grid points on the cell $\mathcal{C}_j$ that contains $y_s$. 
The inner optimization problem~\eqref{copt} for $H_tv$ is a linear program for any fixed $(\bm{x}, \bm{u})$. However, the outer problem~\eqref{nopt} is nonconvex in general.
The nonconvexity originates from the local representation of $y_s = f(\bm{x}, \bm{u}, \hat{\xi}^{[s]})$ by only using the grid points in the cell that contains $y_s$.
However, such a local representation reduces the problem size and 
 allows us to show that $\tilde{T}_t v$ converges uniformly to $T_t  v$ on $\mathcal{Z}_t$ as $\delta$ tends to zero if $v$ is locally Lipschitz continuous on $\mathcal{Z}_{t+1}$.

\begin{proposition}\label{prop:op}
Suppose that Assumptions~\ref{ass:sc} holds, and  that  $v \in \mathbb{B}_b (\mathcal{Z}_{t+1})$ is  locally Lipschitz continuous on $\mathcal{Z}_{t+1}$.
 Then, we have
\[
|(\tilde{T}_t v)(\bm{x}) - (T_t v)(\bm{x})  | \leq   L_v \delta \quad \forall \bm{x} \in \mathcal{Z}_t,
\]
where $L_v := \max_{j = 1,\ldots, N_{\mathcal{C}, t+1}} \sup_{\bm{x}, \bm{x}' \in \mathcal{C}_j: \bm{x} \neq \bm{x}'} \frac{\| v (\bm{x}) - v (\bm{x}')  \|}{\| \bm{x} - \bm{x}'  \|}$, and therefore $\tilde{T}_t v$ converges uniformly to $T_t v$ on $\mathcal{Z}_t$ as $\delta \to 0$.
\end{proposition}

\begin{proof}
Fix an arbitrary $\bm{x} \in \mathcal{Z}_t$.
Under Assumption~\ref{ass:sc}, both \eqref{nopt} and \eqref{copt} have optimal solutions.
Let $\tilde{\bm{u}}$ be an optimal solution of \eqref{nopt}, and let $\tilde{\gamma}$ be a corresponding optimal solution of~\eqref{copt} given $\bm{u} := \tilde{\bm{u}}$.
We also let 
$\tilde{y}_s := f(\bm{x}, \tilde{\bm{u}}, \hat{\xi}^{[s]})$ for each $s$.
Then, we have
\[
\tilde{y}_s = \sum_{i \in \mathcal{N}(\tilde{y}_s)} \tilde{\gamma}_{s,i} \bm{x}^{[i]}.
\]
By the continuity of $v$ on the compact set $\mathcal{Z}_{t+1}$, we have
\[
|  v ( \tilde{y}_s) -  v(\bm{x}^{[i]}) | 
\leq  L_v \| \tilde{y}_s - \bm{x}^{[i]} \| \leq  L_v\delta \quad  \forall i \in \mathcal{N}(\tilde{y}_s).
\]
Since $\sum_{i \in \mathcal{N}(\tilde{y}_s)} \tilde{\gamma}_{s,i} = 1$, we have
\begin{equation}\nonumber
\begin{split}
\bigg |v (\tilde{y}_s) - \sum_{i \in \mathcal{N}(y_s)} \tilde{\gamma}_{s,i} v(\bm{x}^{[i]}) \bigg  | 
&\leq \sum_{i \in \mathcal{N}(\tilde{y}_s)} | \tilde{\gamma}_{s,i} v ( \tilde{y}_s) - \tilde{\gamma}_{s,i} v(\bm{x}^{[i]}) |  \\
& \leq \sum_{i \in \mathcal{N}(\tilde{y}_s)}  \tilde{\gamma}_{s,i} L_v \delta = L_v \delta
\end{split}
\end{equation}
for each $s$.
By using this bound, we obtain that
\begin{equation} \nonumber
\begin{split}
(\tilde{T}_t v) (\bm{x})   &=  r(\bm{x}, \tilde{\bm{u}}) +   \sum_{s=1}^N  \sum_{i \in \mathcal{N}(\tilde{y}_s)} p_s  \tilde{\gamma}_{s,i} v (\bm{x}^{[i]})\\
&\geq r(\bm{x}, \tilde{\bm{u}}) +  \sum_{s=1}^N  p_s v(\tilde{y}_s) -  L_v \delta\\
&= r(\bm{x}, \tilde{\bm{u}}) +  \sum_{s=1}^N p_s v(f(\bm{x}, \tilde{\bm{u}}, \hat{\xi}^{[s]})) -  L_v \delta\\
&\geq (T_t v)(\bm{x}) -  L_v \delta,
\end{split}
\end{equation}
where the last inequality holds because $\tilde{\bm{u}}$ is a feasible solution to the minimization problem in the definition~\eqref{o_opt} of $T_t$.

The other inequality, $(\tilde{T}_t v)(\bm{x}) -   L_v \delta \leq (T_t v)(\bm{x})$, can be shown  by using the second part of the proof of Proposition~\ref{prop:bd}.\footnote{Note that the convexity of $v$ is unused in the second part of the proof of Proposition \ref{prop:bd}. Thus, it is valid in the nonconvex case.}
Since $\bm{x}$ was arbitrarily chosen in $\mathcal{Z}_t$, the result follows. \qed
\end{proof}

Let $\tilde{v}_t: \mathcal{Z}_t \to \mathbb{R}$, $t = 0, \ldots, K$, be defined by
\[
\tilde{v}_t (\bm{x})= (\tilde{T}_t\tilde{v}_{t+1}) (\bm{x}) \quad  \forall \bm{x} \in \mathcal{Z}_t \; \forall t = 0, \ldots, K-1
\]
with $\tilde{v}_K \equiv q$ on $\mathcal{Z}_K$.
By using Proposition \ref{prop:op} and the inductive argument in the proof of Theorem~\ref{thm:bd1},
we can show that $\tilde{v}_t$ converges uniformly to the optimal value function on $\mathcal{Z}_t$ as $\delta$ tends to zero.

\begin{theorem}
Suppose that Assumption \ref{ass:sc} holds, and that $v_t^\star$  is locally Lipschitz continuous on $\mathcal{Z}_{t}$ for each $t =0, \ldots, K$.
Then, we have
\begin{equation}
 |v_t^\star (\bm{x})  - \tilde{v}_t (\bm{x}) | \leq \sum_{k = t+1}^K L_k \delta \quad \forall \bm{x} \in \mathcal{Z}_t \; \forall t =0, \ldots, K,
\end{equation}
where $L_k :=  \sup_{j = 1,\ldots, N_{\mathcal{C},k}} \sup_{\bm{x}, \bm{x}' \in \mathcal{C}_j: \bm{x} \neq \bm{x}'} \frac{\| v_k^\star(\bm{x}) - v_k^\star(\bm{x}')  \|}{\| \bm{x} - \bm{x}'  \|}$,
and therefore
$\tilde{v}_t$ converges uniformly to $v_t^\star$ on $\mathcal{Z}_t$ as $\delta \to 0$.
\end{theorem}

As before, we construct a deterministic Markov policy $\tilde{\pi} := (\tilde{\pi}_0, \ldots, \tilde{\pi}_{K-1})$ by setting $\tilde{\pi}_t(\bm{x})$ to be an optimal solution of \eqref{nopt} with $v:= \tilde{v}_{t+1}$, i.e.,
\begin{equation}\nonumber
\begin{split}
\tilde{\pi}_t (\bm{x}) \in \argmin_{\bm{u} \in \mathcal{U}(\bm{x})} &\big [ r(\bm{x}, \bm{u}) + (H_t  \tilde{v}_{t+1})(\bm{x}, \bm{u}) \big ] \quad \forall \bm{x} \in \mathcal{Z}_{t}.
\end{split}
\end{equation}
Then, the cost-to-go function $v_t^{\tilde{\pi}}: \mathcal{Z}_t \to \mathbb{R}$ under this policy  can be evaluated by solving the following Bellman equation:
\[
v_t^{\tilde{\pi}} := T^{\tilde{\pi}_t} v_{t+1}^{\tilde{\pi}} \quad \forall t = 0, \ldots, K-1,
\]
where $v_K^{\tilde{\pi}} \equiv q$ on $\mathcal{Z}_K$.
This cost $v_t^{\tilde{\pi}}$ incurred by the approximate policy $\tilde{\pi}$
converges uniformly to the optimal value function $v_t^\star$ as the grid resolution becomes finer.

\begin{corollary}
Suppose that Assumption \ref{ass:sc} holds, and that $v_t^\star: \mathcal{Z}_t \to \mathbb{R}$ is  locally Lipschitz continuous  for all $t =0, \ldots, K$.
Then, we have
\[
0 \leq {v}_t^{\tilde{\pi}} (\bm{x}) - v_t^\star (\bm{x})  \leq \sum_{k = t+1}^K 2 L_k \delta \quad \forall \bm{x} \in \mathcal{Z}_t \; \forall t =0, \ldots, K,
\]
where $L_k :=  \sup_{j = 1,\ldots, N_{\mathcal{C},k}} \sup_{\bm{x}, \bm{x}' \in \mathcal{C}_j: \bm{x} \neq \bm{x}'} \frac{\| v_k^\star(\bm{x}) - v_k^\star(\bm{x}')  \|}{\| \bm{x} - \bm{x}'  \|}$, and therefore
${v}_t^{\tilde{\pi}}$ converges uniformly to $v_t^\star$ on $\mathcal{Z}_t$ as $\delta \to 0$.

\end{corollary}
\begin{proof}
By the optimality of $v_t^\star$ under Assumption~\ref{ass:sc}, we have $v_t^\star \leq v_t^{\tilde{\pi}}$ on $\mathcal{Z}_t$ for all $t$.
We now show that $v_t^{\tilde{\pi}} \leq v_t^\star  + \sum_{k=t+1}^K 2 L_k \delta$  on $\mathcal{Z}_t$ by using mathematical induction.
For $t = K$, $v_K^{\tilde{\pi}} = v_K^\star \equiv q$ on $\mathcal{Z}_K$, and thus the induction hypothesis holds. 
Suppose that the induction hypothesis is valid for some $t+1$, i.e.,
\[
0 \leq {v}_{t+1}^{\tilde{\pi}} (\bm{x}) - v_{t+1}^\star (\bm{x})  \leq \sum_{k = t+2}^K 2 L_k \delta \quad \forall \bm{x} \in \mathcal{Z}_{t+1}.
\]
Then, by the monotonicity of ${T}_{t}^{\tilde{\pi}}$, we have
\begin{equation}\label{comp1}
({T}^{\tilde{\pi}_t} v_{t+1}^{\tilde{\pi}})(\bm{x}) \leq ({T}^{\tilde{\pi}_t} v_{t+1}^\star)(\bm{x}) + \sum_{k=t+2}^K 2 L_k \delta \quad \forall \bm{x} \in \mathcal{Z}_{t}.
\end{equation}
Fix an arbitrary state $\bm{x} \in \mathcal{Z}_{t}$.
Under Assumption~\ref{ass:sc}, by the definition of $\tilde{T}_{t}$, there exists $\tilde{\gamma} \in \Delta^N$ such that 
\begin{equation}\label{t1}
(\tilde{T}_{t} v_{t+1}^\star )(\bm{x}) =
r (\bm{x}, \tilde{\pi}_t (\bm{x}) ) 
+  \sum_{s=1}^N \sum_{i \in \mathcal{N}(\tilde{y}_s)} p_s  \tilde{\gamma}_{s,i} v_{t+1}^\star(\bm{x}^{[i]}), 
\end{equation}
where $
\tilde{y}_s := f(\bm{x}, \tilde{\pi}_t(\bm{x}), \hat{\xi}^{[s]})$, 
and
\[
f(\bm{x}, \tilde{\pi}_t (\bm{x}), \hat{\xi}^{[s]}) = \sum_{i \in \mathcal{N}(\tilde{y}_s)} \tilde{\gamma}_{s,i} \bm{x}^{[i]}.
\]
Since $v_{t+1}^\star$ is continuous on $\mathcal{Z}_{t+1}$, we have
\[
| v_{t+1}^\star (\tilde{y}_s) - v_{t+1}^\star (\bm{x}^{[i]})  |
\leq {L}_{t+1} \| \tilde{y}_s - \bm{x}^{[i]} \| \leq {L}_{t+1} \delta \quad \forall i \in \mathcal{N}(\tilde{y}_s).
\]
Note that $\sum_{i \in \mathcal{N}(\tilde{y}_s)} \tilde{\gamma}_{s,i} = 1$. Thus, the following inequalities hold:
\begin{equation} \label{t2}
\begin{split}
\bigg |
v_{t+1}^\star (\tilde{y}_s) - \sum_{i \in \mathcal{N}(\tilde{y}_s)}  \tilde{\gamma}_{s,i} v_{t+1}^\star(\bm{x}^{[i]}) 
\bigg | 
&\leq 
\sum_{i \in \mathcal{N}(\tilde{y}_s)}  \tilde{\gamma}_{s,i}  | v_{t+1}^\star (\tilde{y}_s) - v_{t+1}^\star(\bm{x}^{[i]}) |\\
&\leq 
{L}_{t+1} \delta.
\end{split}
\end{equation}
By \eqref{t1} and \eqref{t2}, we have
\begin{equation} \label{comp2} 
\begin{split}
(\tilde{T}_{t} v_{t+1}^\star )(\bm{x}) 
&\geq 
r (\bm{x}, \tilde{\pi}_t (\bm{x}) ) 
+  \sum_{s=1}^N p_s  v_{t+1}^\star (\tilde{y}_s) - {L}_{t+1} \delta\\
&=r (\bm{x}, \tilde{\pi}_t (\bm{x}) ) 
+ \sum_{s=1}^N p_s  v_{t+1}^\star (f(\bm{x}, \tilde{\pi}_t(\bm{x}), \hat{\xi}^{[s]})) - {L}_{t+1} \delta\\
&\geq (T^{\tilde{\pi}_t} v_{t+1}^\star )(\bm{x}) - {L}_{t+1} \delta.
\end{split}
\end{equation}
On the other hand, by Proposition~\ref{prop:op}, we have
\begin{equation}\label{comp3}
(\tilde{T}_{t} v_{t+1}^\star)(\bm{x}) \leq (T_{t} v_{t+1}^\star)(\bm{x}) + L_{t+1} \delta.
\end{equation}
Since $\bm{x}$ was arbitrarily chosen in $\mathcal{Z}_{t}$, by combining inequalities \eqref{comp1}, \eqref{comp2} and \eqref{comp3}, we conclude that
\[
v_{t}^{\tilde{\pi}} = {T}^{\tilde{\pi}_t} v_{t+1}^{\tilde{\pi}} \leq {T}_t v_{t+1}^\star + \sum_{k={t+1}}^K 2 L_k \delta = v_{t}^\star +\sum_{k={t+1}}^K 2 L_k \delta
\]
on $\mathcal{Z}_{t}$,
which implies that the induction hypothesis is valid for $t$ as desired.  \qed
\end{proof}

From the computational perspective, it is challenging to obtain a globally optimal solution of \eqref{nopt} due to nonconvexity.
Thus, over-approximating $\tilde{T}_{t} v$ is inevitable in practice, and the quality of an approximate policy $\tilde{\pi}$ from an over-approximation of $\tilde{v}_t$'s depends on the quality of locally optimal solutions to \eqref{nopt} evaluated in the process of value iteration. 
As in Section~\ref{sec:ca}, one may be able to characterize a suboptimality bound for the approximate policy, for example, by using a convex relaxation of \eqref{nopt}.
In fact, the optimization problem \eqref{con_opt} can be interpreted as a convex relaxation of \eqref{nopt} in the case of control-affine systems.
Another classical and possibly practical method for solving the minimization problem in \eqref{nopt} is to discretize the action space. 
After discretization, we evaluate the optimal value of the program \eqref{copt} for each action, and then compare the optimal values to approximately solve \eqref{nopt}. This approach is used in solving the nonconvex problem in Section~\ref{num:ca}.

On the other hand, the inner optimization problem~\eqref{copt} can be efficiently solved using existing linear programming algorithms such as interior point and simplex methods~(e.g., \cite{Dantzig1998, Bertsimas1997}). 
The linear optimization problem has an equivalent form with
the reduced optimization variable $\tilde{\gamma} := (\gamma_i)_{i \in \mathcal{N}_y}$ instead of using the full vector $\gamma$.
Using this reduced optimization problem significantly gains computational efficiency since the size of the reduced optimization variable is independent of grid resolution. 
For example, in the case of regular Cartesian grid with size $N_x^n$, the number of optimization variables decreases
from $N_x^n \times N$ to $2^n \times N$.

\section{Numerical Experiments}\label{sec:num}

\subsection{\bf Linear-Convex Control}\label{num:lc}

We first demonstrate the performance of the proposed method through a  linear-convex stochastic control problem. 
Consider the linear stochastic system
\[
x_{t+1} = 
{A} x_t + B u_t + C \xi_t,
\]
where $x_t \in \mathbb{R}^2$ and $\xi_t \in \mathbb{R}$.
We set $A = \begin{bmatrix} 0.85 & 0.1 \\0.1 & 0.85 \end{bmatrix}$,   and $C = \begin{bmatrix} 1 \\ 1 \end{bmatrix}$.
To demonstrate that our method can handle high-dimensional action spaces, 
the control variable $u_t$ is chosen to be 1000 dimensional.  The matrix $B$ is an $2$ by $1000$ matrix of all entries being sampled independently from the uniform distribution over $[0,1]$ and then normalized so that the 1-norm of each row is~1.\footnote{The matrix $B$ used in our experiments can be downloaded from the following link: \url{http://coregroup.snu.ac.kr/DB/B1000.mat}.} 
The stage-wise cost function is given by
\[
r(x_t, u_t) = \|x_t\|_1 + \|u_t\|_2^2,
\]
which is convex but non-quadratic. 
The set of admissible control actions is chosen as
\[
u_t \in \mathcal{U} := [-0.15, 0.15]^{1000}.
\]
The support elements of $\xi_t$ are sampled according to the uniform distribution over $[-0.1,0.1]$.
The computational domains are chosen as follows: $\mathcal{Z}_t := [-1 - 0.2t, 1 + 0.2t]^2$ for $t = 0, \ldots, 5$ so that $\mathcal{Z}_0 =[-1, 1]^2$ and 
$\mathcal{Z}_{5} = [-2,2]^2$. This choice of computational domains satisfies that
\[
\big \{
A \bm{x} + B \bm{u} + C\bm{\xi} : \bm{x} \in \mathcal{Z}_t, \bm{u} \in [-0.15, 0.15]^2, \bm{\xi} \in [-0.1, 0.1]  \big \} \subseteq \mathcal{Z}_{t+1},
\]
and $\mathcal{Z}_0 \subseteq \mathcal{Z}_1 \subseteq \cdots \subseteq \mathcal{Z}_{5}$.
Each $\mathcal{Z}_t$ is discretized as a two-dimensional regular Cartesian grid using the approach in Appendix~\ref{app:grid}
with grid points $\{ (-1 -0.2t + \delta_x (i-1), -1 -0.2t + \delta_x (j-1))  : i,j = 1, \ldots,  (2 + 0.4 t)/\delta_x + 1\}$ with equal grid spacing $\delta_x$.
The terminal cost is set to be $q \equiv 0$.
The numerical experiments  were conducted on a PC with 4.2 GHz Intel Core i7 and 64GB RAM. The optimization problems were solved using the interior point method in CPLEX.

\begin{figure}[tb]
\begin{center}
\includegraphics[width=3.3in]{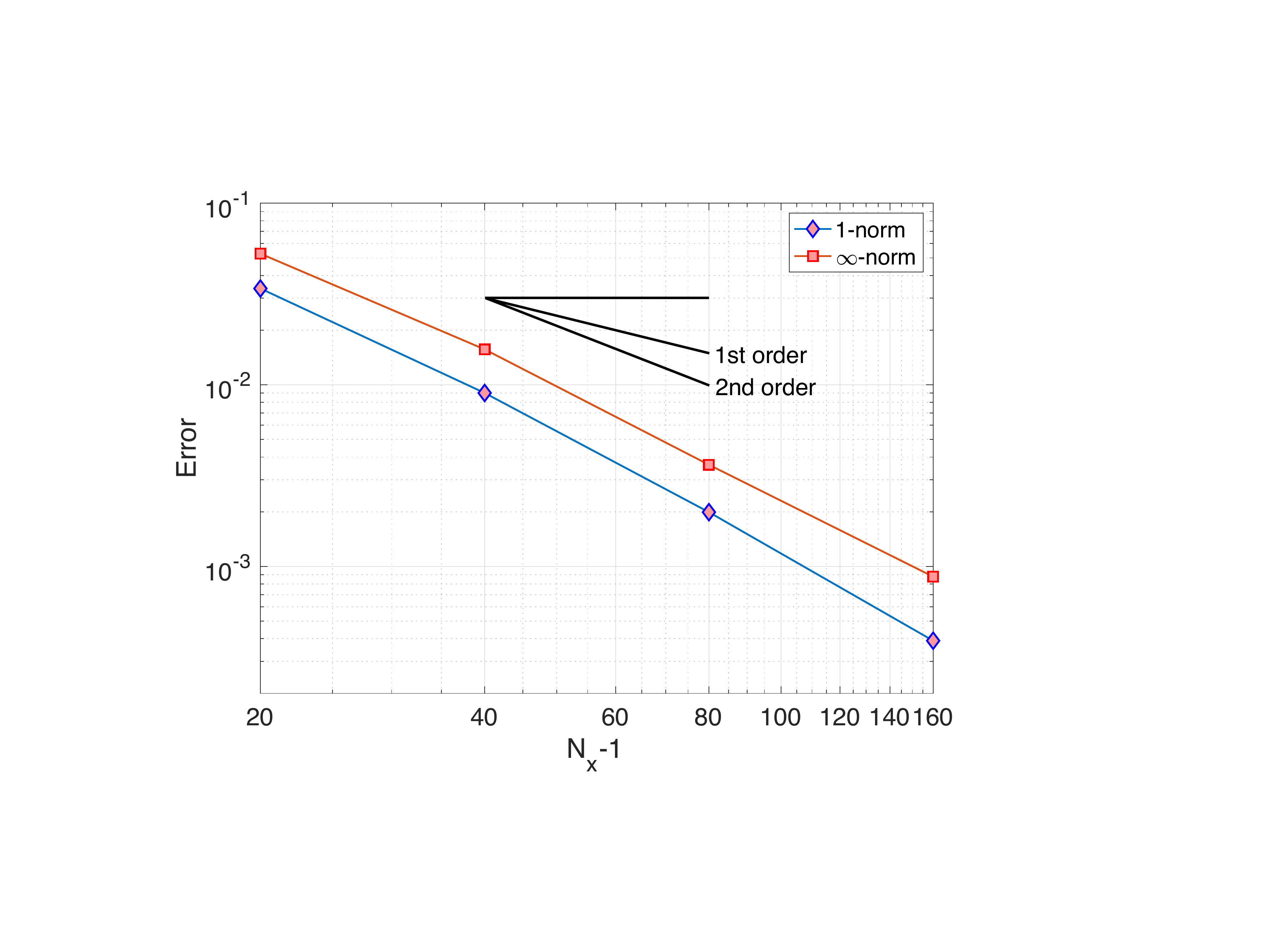}   
\caption{Empirical convergence rates of $\hat{v}_{(\delta_x),0}$ for the linear convex problem with grid size $21^2, 41^2, 81^2$ and $161^2$. Here, the label $N_x$ denotes the number of grid points on one axis. The errors $\| \hat{v}_{(N_x),0} - \bar{v}_0 \|$  are measured by the $1$($\Diamond$)-- and $\infty$($\Box$)--norm over $\mathcal{Z}_0 = [-1, 1]^2$.}  
\label{fig:grid_convergence}    
\end{center}           
\end{figure}

\begin{table}[tb]
\small
\centering
\caption{Computation time (in seconds) and errors for the linear $L^1$-control problem with different grid sizes (evaluated over $\mathcal{Z}_0 = [-1, 1]^2$).\label{tab:1}}
\begin{tabular}{ l | >{\centering}p{0.5in}>{\centering}p{0.5in}  >{\centering}p{0.5in}  >{\centering\arraybackslash}p{0.5in} }
\# of grid points       & $21^2$ & $41^2$ & $81^2$ & $161^2$  \\
\hline \hline
Time (sec)  & 18.22  &  139.93  & 1877.25   &  43736.92  
   \\
 \hline
$\ell_1$-error   & 0.0339     &   0.0090   &  0.0020    &   0.0004   \\
 \hline
 $\ell_\infty$-error   &  0.0527   &   0.0157   &  0.0036   &   0.0009    \\
 \hline
\end{tabular}
\end{table}

\begin{figure}[tb]
\begin{center}
\includegraphics[width=3.5in]{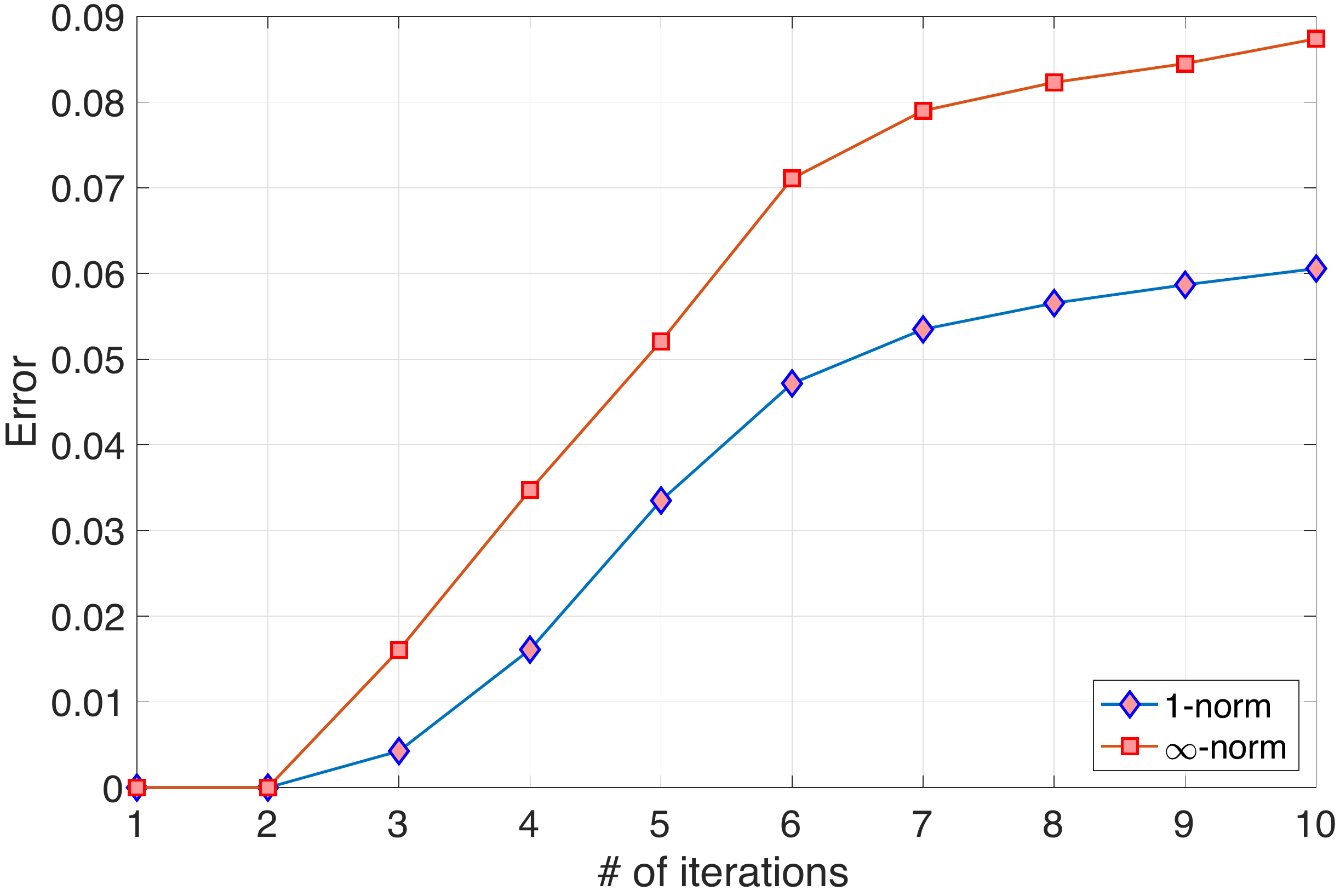}   
\caption{Errors for the linear convex problem with respect to the number of DP iterations. The errors are measured by the $1$($\Diamond$)-- and $\infty$($\Box$)--norm over $\mathcal{Z}_0 = [-1, 1]^2$. }  
\label{fig:error_time}    
\end{center}           
\end{figure}

Let $N_x$ denote the number of grid points on one axis of $\mathcal{Z}_5$. Thus, the total number of grid points is $N_x^2$.
To demonstrate the convergence of the proposed method as $\delta = \sqrt{2}\delta_x \to 0$ or equivalently as $N_x \to \infty$, we fix the number of sample data as $N = 10$, and compute $\hat{v}_{(N_x),0}$ for $N_x = 21, 41, 81, 161$ or equivalently for $\delta_x = 0.2, 0.1, 0.05, 0.025$, where 
$\hat{v}_{(N_x), t}$ denotes the approximate value function at $t$ with the number of grid points being $N_x^2$.
Setting $\bar{v}_0 := \hat{v}_{(N_x = 321), 0}$, the errors $\| \hat{v}_{(N_x),0} - \bar{v}_0\|$  over $\mathcal{Z}_0 = [-1, 1]^2$, measured by the $1$-- and $\infty$--norm,  are shown in Fig.~\ref{fig:grid_convergence}. 
The CPU time for computation is also reported in Table~\ref{tab:1}.\footnote{
The CPU time increases superlinearly with the number of grid points. This is because the size of the optimization problem~\eqref{con_opt} also increases with the grid size. 
Note that the problem size is invariant
 when using the bi-level method in Section~\ref{sec:gen}.
 Thus, in that case the CPU time scales linearly as shown in Table~\ref{tab:d}.
}
In this example, the empirical convergence rate of our method is beyond the second order.\footnote{The observation of the second-order empirical convergence rate  is consistent with our theoretical result since Theorem 3.1 only suggests that the suboptimality gap decreases with the first-order rate. Thus, the actual convergence rate can be higher than the convergence rate for the suboptimality gap.} 
Furthermore, the relative error $\| (\hat{v}_{(N_x), 0} - \bar{v}_0)/\bar{v}_0 \|_1$ over $\mathcal{Z}_0 = [-1, 1]^2$ is $0.02\%$ in the case of $N_x = 161$.

Fig.~\ref{fig:error_time} shows the errors  in $1$-- and $\infty$--norm, evaluated over $\mathcal{Z}_0 = [-1. 1]^2$ 
depending on  the number of DP iterations or equivalently the length of time horizon when $\delta_x = 0.2$.
The errors grow sublinearly with respect to the number of DP iterations. 
This result is consistent with the error bound in Theorem~\ref{thm:bd1}, assuming that the constant $L_k$ does not change much over time.

To test the empirical convergence of the proposed method as the sample size increases, we also compute $\hat{v}_{(N),0}$ for $N=20, 40, 80, 160$ with $\delta_x$ fixed as $0.2$, where $\hat{v}_{(N),0}$ denotes the approximate value function at $t=0$ with sample size $N$.
The errors $\| \hat{v}_{(N),0} - \tilde{v}_0\|$  over $\mathcal{Z}_0 = [-1, 1]^2$ with $N=20, 40, 80, 160$ are shown in Fig.~\ref{fig:SAA_convergence}, where $\tilde{v}_0:= \hat{v}_{(N = 320), 0}$.  The CPU time   is reported in Table~\ref{tab:saa}.
In this example, the empirical convergence rate   is below the first order. 
The relative error $\| (\hat{v}_{(N), 0} - \tilde{v}_0)/\tilde{v}_0 \|_1$ over $\mathcal{Z}_0 = [-1, 1]^2$ is $0.72\%$ in the case of $N = 160$.

\begin{figure}[tb]
\begin{center}
\includegraphics[width=3.5in]{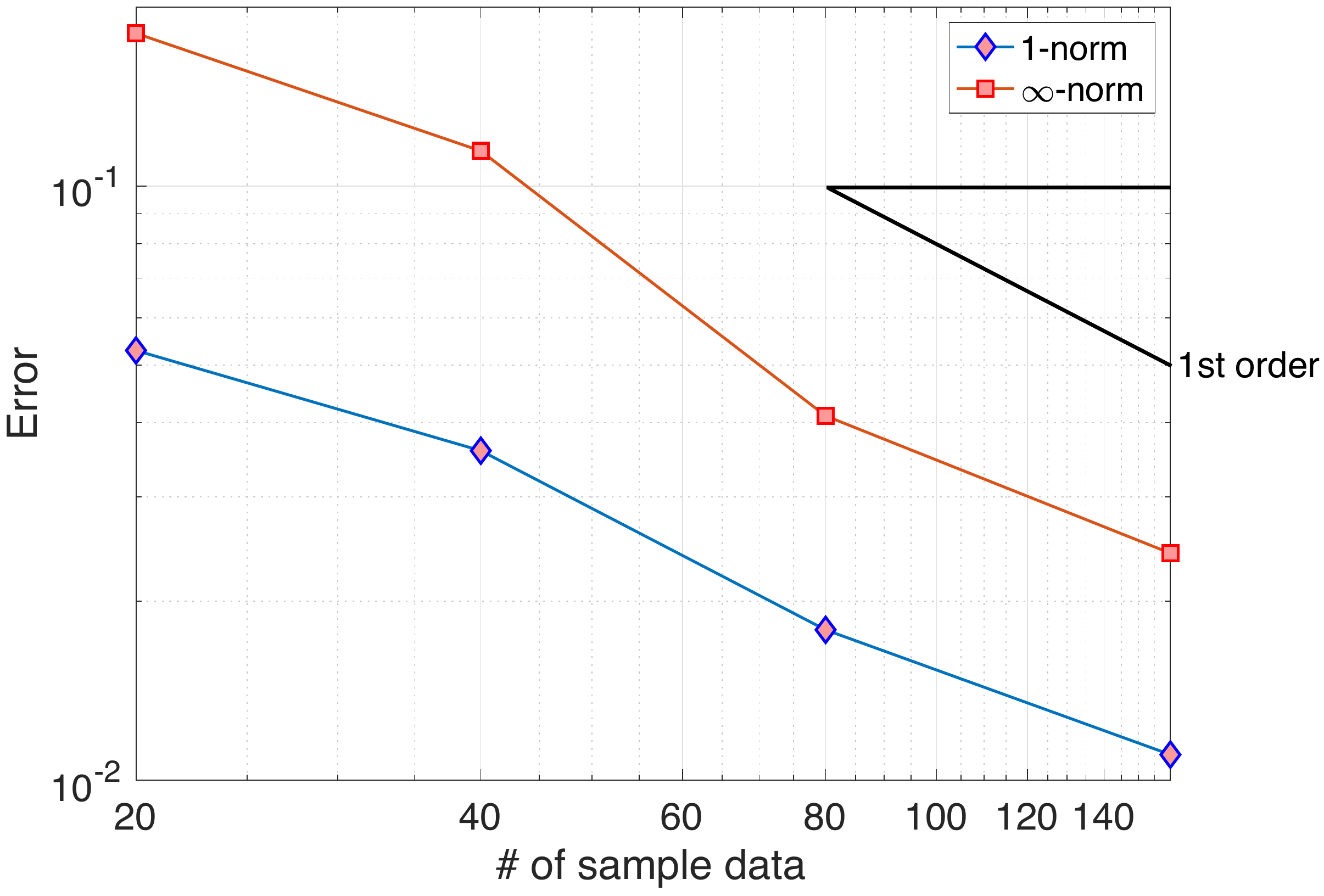}   
\caption{Empirical convergence rates of $\hat{v}_{(N),0}$ for the linear convex problem with sample size $N = 20, 40, 80, 160$. The errors $\| \hat{v}_{(N),0} - \tilde{v}_0 \|$  are measured by the $1$($\Diamond$)-- and $\infty$($\Box$)--norm  over $\mathcal{Z}_0 = [-1, 1]^2$.}  
\label{fig:SAA_convergence}    
\end{center}           
\end{figure}

\begin{table}[tb]
\small
\centering
\caption{Computation time (in seconds) and errors for the linear $L^1$-control problem with different sample sizes   (evaluated over $\mathcal{Z}_0 = [-1, 1]^2$).\label{tab:saa}}
\begin{tabular}{l | >{\centering}p{0.5in}>{\centering}p{0.5in}  >{\centering}p{0.5in}  >{\centering\arraybackslash}p{0.5in} }
Sample size      & $20$ & $40$ & $80$ & $160$      \\
\hline \hline
Time (sec)  & 35.81    & 72.53    &227.79    &  857.14  
  \\
 \hline
$\ell_1$-error  &  0.0529  &   0.0359  &  0.0179  &  0.0110   \\
 \hline
$\ell_\infty$-error  & 0.1808    &0.1146  & 0.0410   & 0.0241      \\
 \hline
\end{tabular}
\end{table}

\subsection{\bf Control-Affine Nonlinear System}\label{num:ca}

We consider the following nonlinear, control-affine system:
\[
x_{t+1} = x_t + [w_t  (1- x_t ) x_t  -  x_t u_t  ] \Delta t,
\]
which models the outbreak of an infectious disease~\cite{Sethi2000}.
Here, the system state $x_t \in [0,1]$ represents the ratio of infected people in a given population, the control input $u_t \in [0, u_{\max}]$ represents a public treatment action, and $w_t > 0$ denotes infectivity.
The support of $w_t$ was sampled from the uniform distribution over $[1,2]$.
The stage-wise cost function is chosen as the following convex, non-quadratic function:
\[
r(x_t, u_t) = x_t + \lambda u_t^2, 
\]
where $\lambda > 0$ is the intervention cost.
We choose $u_{\max} = 1$, $\Delta t = 10^{-1}$, $K = 20$, $\lambda = 10^{-1}$,  $N = 10$, and $q \equiv 0$.
In this case, the system state remains in the domain $[0,1]$ for all time. 
Thus, we set $\mathcal{Z}_t \equiv [0,1]$ for all $t = 0, \ldots, 20$. 
The domain $\mathcal{Z}_t$ is discretized with a uniform grid with grid spacing $\delta_x = 0.01$. Thus, the number of grid points is $101$. In this setting, the CPU time for running our method is 12.57 seconds.

\begin{figure}[tb]
\begin{center}
\includegraphics[width=5.5in]{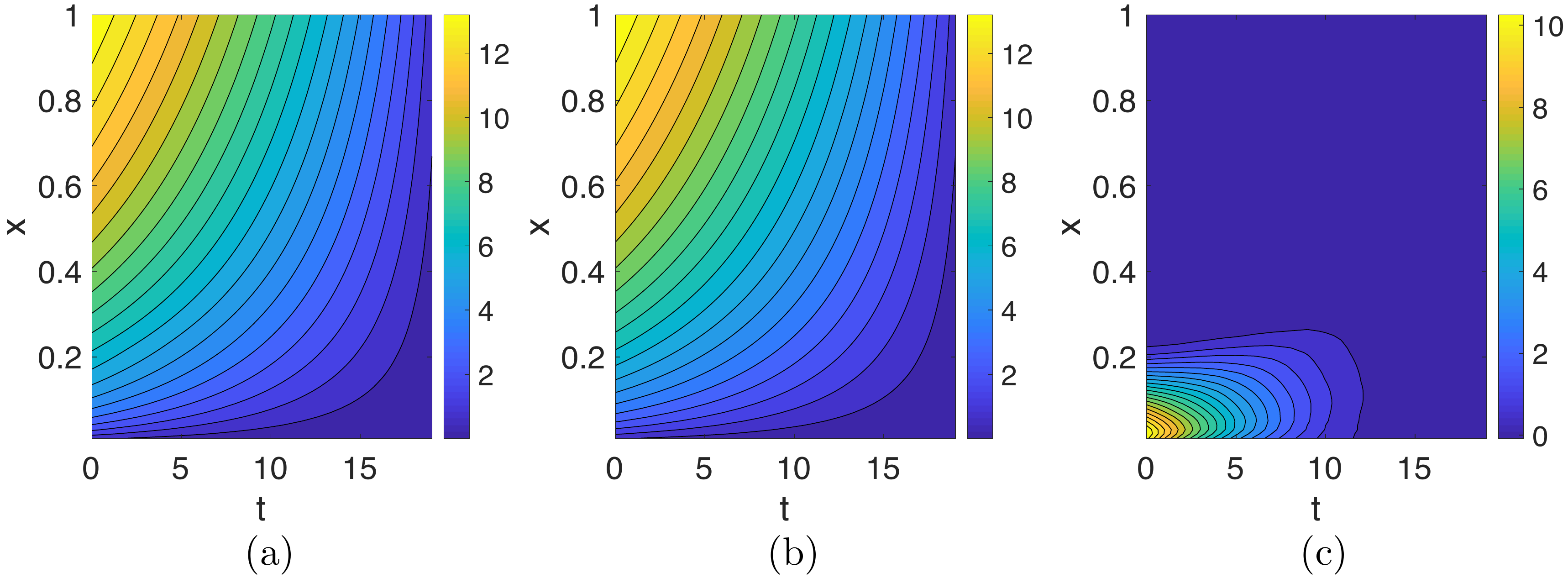}   
\caption{(a) The value function $v_t^{\hat{\pi}} (\bm{x})$ of the approximate policy $\hat{\pi}$, (b) the optimal value function $v_t^\star (\bm{x})$, and (c)
the relative error  $\frac{v_t^{\hat{\pi}} (\bm{x}) - v_t^\star (\bm{x})}{v_t^\star (\bm{x})} \times 100$\% for the epidemic control problem.}  
\label{fig:epidemic_error}    
\end{center}           
\end{figure}

To examine the suboptimality of our method in this problem with a control-affine system, we compare the value function $v_t^{\hat{\pi}}$ of our approximate policy and the optimal value function $v_t^\star$.\footnote{To compute the optimal value function, we used the method in Section \ref{sec:gen} discretizing the action space with $1001$ equally spacing grid points.} 
As shown in Fig.~\ref{fig:epidemic_error} (a) and (b), $v_t^{\hat{\pi}}$ and $v_t^\star$ look very similar to each other. Fig.~\ref{fig:epidemic_error} (c) shows the relative error 
\[
\frac{v_t^{\hat{\pi}} (\bm{x}) - v_t^\star (\bm{x})}{v_t^\star (\bm{x})} \times 100 \%.
\]
The relative error is less than 11\% for all $\bm{x} \in [0,1]$ and $t =0, \ldots, 20$. 
The $\ell_1$-norm and the $\ell_2$-norm of the relative error are 0.45\% and $1.44\%$, respectively. 
This result confirms the utility of our method even in nonconvex problems with control-affine systems.

\subsection{\bf Fully Nonlinear System}\label{num:nl}

We consider the following discrete-time Dubins car model~\cite{Dubins1957}, which is fully nonlinear:
\begin{equation} \nonumber
\begin{split}
\mathrm{x}_{t+1} &= \mathrm{x}_t + v_t \cos \theta_t, \quad \mathrm{y}_{t+1} = \mathrm{y}_t + v_t \sin \theta_t, \quad \theta_{t+1} = \theta_t + \frac{1}{L} v_t \tan s_t,
\end{split}
\end{equation}
where $x_t := (\mathrm{x}_t, \mathrm{y}_t, \theta_t) \in \mathbb{R}^3$ is the system state, and $u_t := (v_t, s_t) \in \mathbb{R}^2$ is the control action.
Specifically, $(x_t, y_t)$ represents the position of the vehicle in $\mathbb{R}^2$, and $\theta_t$ denotes the heading angle. 
Moreover, $v_t$ and $s_t$ denote a velocity input and a steering input, respectively. 
The third state equation specifies the turning rate.
The Dubins vehicle model is commonly used in robotics and controls as a model for describing the planar motion of wheeled vehicles. 
We set $v_t \in \{0, -0.1\}$, $s_t \in \{-1, 0, 1\}$, and $L = 0.5$.
The control objective is to steer the vehicle to the y-axis with heading angle $\pi$. Accordingly, the terminal cost function is set to be 
\[
q(\bm{x}) =\| (\bm{x}_2, \bm{x}_2) - (0, \pi)\|^2,
\]
 where $\bm{x} = (\bm{x}_1, \bm{x}_2, \bm{x}_3) \in \mathbb{R}^3$ with no running costs. 
With $K = 20$, the computational domains are chosen as follows:
\[
\mathcal{Z}_t := [-0.5 - 0.1t, 0.5+ 0.1t] \times [-0.5 - 0.1t, 0.5+ 0.1t] \times [0, 2\pi]
\]
 for $t =  0, 1, \ldots, 20$. It is straightforward to check that this choice satisfies the condition~\eqref{inclusion}.
Each $\mathcal{Z}_t$ is discretized as a three-dimensional Cartesian grid
using the approach in Appendix~\ref{app:grid} with grid points
\begin{equation}\nonumber
\begin{split}
\big \{ &(-0.5 -0.1t + \delta_x (i-1), -0.5 -0.1 t + \delta_x (j-1), 0 + \delta_\theta (k-1)): \\
 & i, j = 1, \ldots, (1 + 0.2 t)/\delta_x + 1, k = 1, \ldots, 2\pi/\delta_\theta + 1 \big \}.
\end{split}
\end{equation}
In our numerical experiment, we set $\delta_x = 0.1$ and $\delta_\theta = \pi/20$, and thus $\mathcal{Z}_{20}$ is discretized with a grid of size $51 \times 51 \times 41$. 
The CPU time for running our method  in Section~\ref{sec:gen} is 4630.93 seconds. 
Fig.~\ref{fig:dubin} shows the resulting vehicle trajectories starting from $(\mathrm{x}_0, \mathrm{y}_0) = (0,-0.5)$ with four different initial heading angles $\theta_0 = 0, \pi/2, \pi, 3\pi/2$. 
In all four cases, the vehicle moves to the y-axis with heading $\pi$ as desired.

\begin{figure}[tb]
\begin{center}
\includegraphics[width=3.5in]{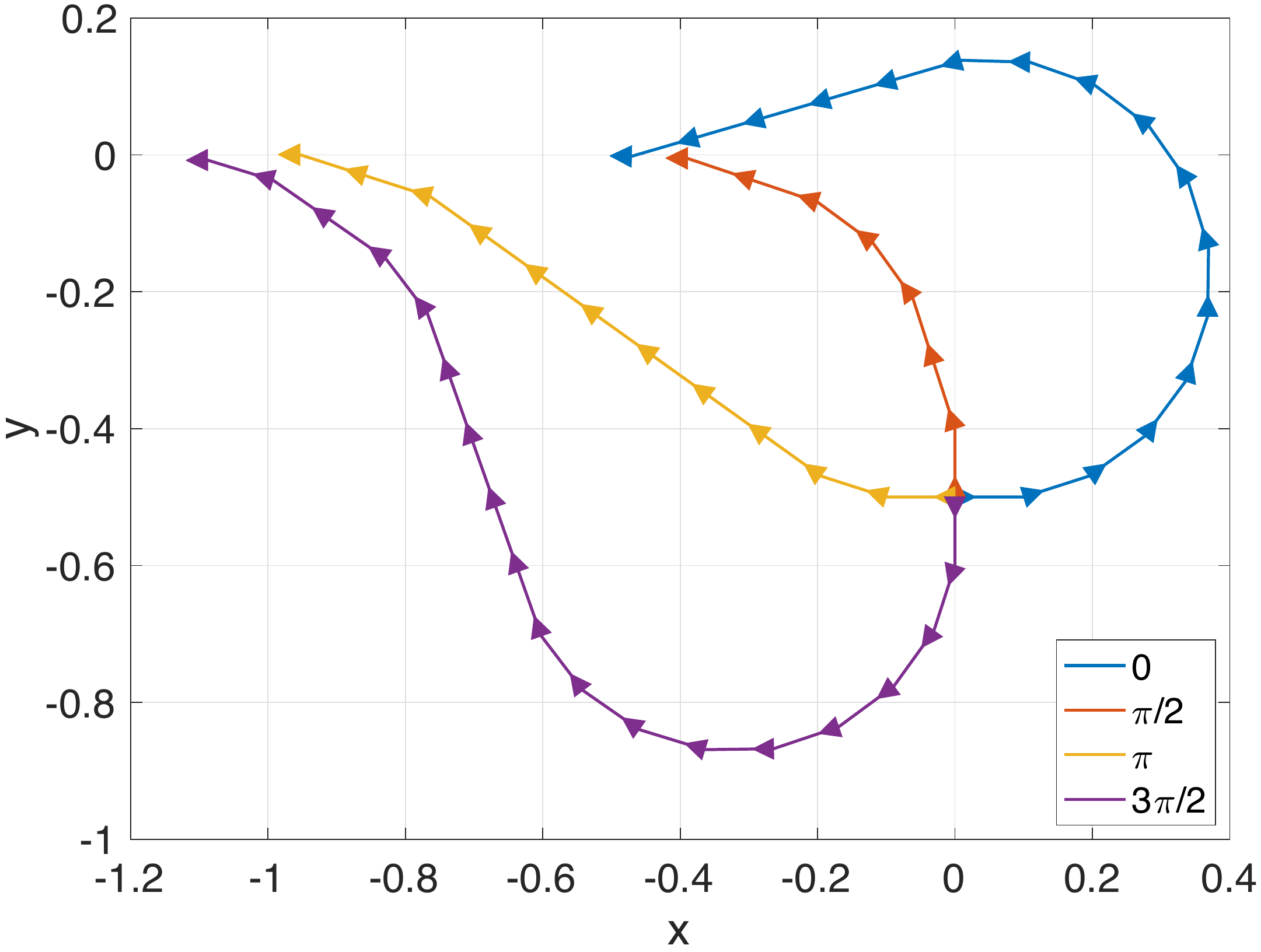}   
\caption{Controlled trajectories for the Dubins vehicle model with different initial heading angles.  Each arrow represents the vehicles's heading at a particular instance.}
\label{fig:dubin}    
\end{center}           
\end{figure}

As in Section~\ref{num:lc}, let $N_x$ denote the number of grid points on one axis of $\mathcal{Z}_5$. Thus, the total number of grid points is $N_x^3$.
To demonstrate the convergence of our method as $\delta_x, \delta_\theta \to 0$ or equivalently as $N_x \to \infty$, 
we evaluate the errors $\| \hat{v}_{(N_x),0} - \bar{v}_0 \|$ in the $1$-- and $\infty$--norm over $\mathcal{Z}_0$, where $\hat{v}_{(N_x), t}$ denotes the approximate value function at $t$ with the number of grid points being $N_x^3$, and  $\bar{v}_0 := \hat{v}_{(N_x = 161), 0}$.
Table~\ref{tab:d} shows the errors and the CPU time with different grid sizes and $K = 5$. 
This result indicates that the empirical convergence rate in this example is approximately the second order. 
Another important observation is that the computation time scales linearly with the number of grid points. This is because the size of the linear optimization problem \eqref{copt} is invariant with the grid size as discussed in Section~\ref{sec:gen}.

\begin{table}[tb]
\small
\centering
\caption{Computation time (in seconds) and errors for the Dubins vehicle  problem with different grid sizes (evaluated over $\mathcal{Z}_0$).\label{tab:d}}
\begin{tabular}{ l | >{\centering}p{0.7in}   >{\centering}p{0.7in}  >{\centering\arraybackslash}p{0.7in} }
\# of grid points       & $21^3$ & $41^3$ & $81^3$   \\
\hline \hline
Time (sec)  & 135.31   &  975.54  & 7574.77 
   \\
 \hline
$\ell_1$-error   & 0.0082    &   0.0027  &  0.0008      \\
 \hline
 $\ell_\infty$-error   & 0.0122    &0.0038   &0.0009         \\
 \hline
\end{tabular}
\end{table}

\section{Conclusions}

We have proposed and analyzed a convex optimization-based method designed to solve dynamic programs in continuous state and action spaces.  
This interpolation-free approach provides a control policy of which performance converges uniformly to the optimum as the computational grid becomes finer when the optimal and approximate value functions are convex.
In the case of control-affine systems with convex cost functions,
the proposed bound on the gap between 
 the cost-to-go function of the approximate policy 
 and the optimal value function  is useful to gauge  the degree of suboptimality.
 In general nonlinear cases, a simple modification to a bi-level optimization formulation, of which inner problem is a linear program,  maintains the uniform convergence properties if an optimal solution to this bi-level program can be computed.


\appendix

\section{State Space Discretization Using a Rectilinear Grid}\label{app:grid}

In this appendix, we provide a concrete way to discretize the state space using a rectilinear grid. The construction below satisfies all the conditions in Section~\ref{sec:dis}.

\begin{enumerate}

\item Choose a convex compact set $\mathcal{Z}_0 := [\underline{\bm{x}}_{0, 1}, \overline{\bm{x}}_{0, 1}] \times [\underline{\bm{x}}_{0, 2}, \overline{\bm{x}}_{0, 2}] \times \cdots \times [\underline{\bm{x}}_{0, n}, \overline{\bm{x}}_{0, n}]$, and discretize it using an $n$-dimensional rectilinear grid. 
Set $t \leftarrow 0$.

\item 
Compute (or over-approximate) the forward reachable set\footnote{The forward reachable set can be over-approximated in an analytical way, particularly when a loose approximation is allowed. For a high quality of approximation, one may use advanced computational techniques with semidefinite approximation~\cite{Magron2019} and ellipsoidal approximation~\cite{Kurzhanskiy2011}, among others.}
\[
R_{t} := \big \{ f(\bm{x}, \bm{u}, \bm{\xi}) : \bm{x} \in \mathcal{Z}_{t}, \bm{u} \in \mathcal{U}(\bm{x}), \bm{\xi} \in \Xi \big \}. 
\]

\item 
Choose a convex compact set $\mathcal{Z}_{t+1} :=  [\underline{\bm{x}}_{t+1, 1}, \overline{\bm{x}}_{t+1, 1}] \times [\underline{\bm{x}}_{t+1, 2}, \overline{\bm{x}}_{t+1, 2}] \times \cdots \times [\underline{\bm{x}}_{t+1, n}, \overline{\bm{x}}_{t+1, n}]$ such that  $R_t \subseteq \mathcal{Z}_{t+1}$. 

\item 
Expand the rectilinear grid to fit  $\mathcal{Z}_{t+1}$.

\item Stop if $t+1 = K$; otherwise, set $t \leftarrow t+1$ and go to Step 2. 

\end{enumerate}

We can then choose $\mathcal{C}_i$ as each grid cell. 
We label $\mathcal{C}_i$ so that $\bigcup_{i=1}^{N_{\mathcal{C}, t}} \mathcal{C}_i = \mathcal{Z}_t$
for all $t$. A two-dimensional example is shown in Fig.~\ref{fig:grid}.
This construction approach was used in Section \ref{num:lc} and \ref{num:nl}.

\section{Proof of Lemma~\ref{lem:conv}}\label{app:conv}

\begin{proof}
Suppose that $v: \mathcal{Z}_{t+1} \to \mathbb{R}$ is convex.
 Fix two arbitrary states $\bm{x}_{(k)} \in \mathcal{Z}_{t}$, $k=1,2$.
 
 We first show that $T_t v$ is convex on $\mathcal{Z}_t$.
Under Assumption~\ref{ass:sc}, for $k=1,2$,
there exists $\bm{u}_{(k)} \in \mathcal{U}(\bm{x}_{(k)})$ such that
\[
(T_t v)(\bm{x}_{(k)}) = r(\bm{x}_{(k)}, \bm{u}_{(k)}) 
+   \sum_{s=1}^N p_s  v (f(\bm{x}_{(k)}, \bm{u}_{(k)}, \hat{\xi}^{[s]}))
\]
and $f(\bm{x}_{(k)}, \bm{u}_{(k)}, \hat{\xi}^{[s]}) \in \mathcal{Z}_{t+1}$ for all $s \in \mathcal{S}$.
Fix an arbitrary $\lambda \in (0,1)$, and let 
$\bm{x}_{(\lambda)} := \lambda \bm{x}_{(1)} + (1-\lambda) \bm{x}_{(2)} \in \mathcal{Z}_t$ and $\bm{u}_{(\lambda)} := \lambda \bm{u}_{(1)} + (1-\lambda) \bm{u}_{(2)}$.
By Assumption~\ref{ass:lc}, we have $\bm{u}_{(\lambda)} \in \mathcal{U}(\bm{x}_{(\lambda)})$. 
Thus, the following inequality holds:
\begin{equation} \nonumber
\begin{split}
(T_t v)(\bm{x}_{(\lambda)}) \leq r(\bm{x}_{(\lambda)}, \bm{u}_{(\lambda)}) + 
   \sum_{s=1}^N  p_s v(f(\bm{x}_{(\lambda)}, \bm{u}_{(\lambda)}, \hat{\xi}^{[s]})).
\end{split}
\end{equation}
Since $(\bm{x}, \bm{u}) \mapsto r(\bm{x}, \bm{u}) $ and $(\bm{x}, \bm{u}) \mapsto v ( f(\bm{x}, \bm{u}, \bm{\xi}))$ are convex for each $\bm{\xi} \in \Xi$, we have
\begin{equation}\nonumber
\begin{split}
(T_t v)(\bm{x}_{(\lambda)}) &\leq \lambda r(\bm{x}_{(1)}, \bm{u}_{(1)}) + (1-\lambda) r (\bm{x}_{(2)}, \bm{u}_{(2)})\\
& + \sum_{s=1}^N p_s [
\lambda v(f(\bm{x}_{(1)}, \bm{u}_{(1)}, \hat{\xi}^{[s]})) + (1-\lambda) v(f(\bm{x}_{(2)}, \bm{u}_{(2)},  \hat{\xi}^{[s]}))
].
\end{split}
\end{equation}
Therefore, we obtain that
\[
(T_t v)(\bm{x}_{(\lambda)}) \leq \lambda (T_t v)(\bm{x}_{(1)}) +  (1-\lambda) (T_t v)(\bm{x}_{(2)}).
\]
which implies that $T_t v$ is convex on $\mathcal{Z}_t$.

Next we show that $\hat{T}_t v$ is convex on $\mathcal{Z}_t$.
Under Assumption~\ref{ass:sc}, for $k = 1, 2$,
there exists an optimal solution $(\hat{\bm{u}}_{(k)}, \hat{\gamma}_{(k)}) \in \mathcal{U}(\bm{x}_{(k)}) \times \Delta^N$  to \eqref{con_opt} with $\bm{x} := \bm{x}_{(k)}$, i.e.,
\[
(\hat{T}_t v)(\bm{x}_{(k)}) = r(\bm{x}_{(k)}, \hat{\bm{u}}_{(k)}) + \sum_{s=1}^N \sum_{i=1}^{M_{t+1}} p_s \hat{\gamma}_{(k),s,i} v (\bm{x}^{[i]})
\]
and $f(\bm{x}_{(k)}, \hat{\bm{u}}_{(k)}, \hat{\xi}^{[s]}) = \sum_{i=1}^{M_{t+1}}  \hat{\gamma}_{(k),s,i} \bm{x}^{[i]} \in \mathcal{Z}_{t+1}$ for all $s \in \mathcal{S}$.
Given any fixed $\lambda \in ]0,1[$, let $(\hat{\bm{u}}_{(\lambda)}, \hat{\gamma}_{(\lambda)}) := \lambda (\hat{\bm{u}}_{(1)}, \hat{\gamma}_{(1)}) + (1-\lambda) (\hat{\bm{u}}_{(2)}, \hat{\gamma}_{(2)})$.
Since $(\bm{x}, \bm{u}) \mapsto f(\bm{x}, \bm{u}, \hat{\xi}^{[s]})$ is affine on $\mathcal{K}$, 
\[
f(\bm{x}_{(\lambda)}, \hat{\bm{u}}_{(\lambda)}, \hat{\xi}^{[s]})
 = \sum_{i=1}^{M_{t+1}} \hat{\gamma}_{(\lambda),s,i} \bm{x}^{[i]}.
\]
This implies that $(\hat{\bm{u}}_{(\lambda)}, \hat{\gamma}_{(\lambda)})$ is a feasible solution to the minimization problem in the definition of $(\hat{T}_t v)(\bm{x}_{(\lambda)})$. 
Therefore,
\begin{equation} \nonumber
\begin{split}
(\hat{T}_t v)(\bm{x}_{(\lambda)}) &\leq r(\bm{x}_{(\lambda)}, \hat{\bm{u}}_{(\lambda)}) +  \sum_{s=1}^N \sum_{i=1}^{M_{t+1}} p_s \hat{\gamma}_{(\lambda),s,i} v (\bm{x}^{[i]})\\
&\leq \lambda r(\bm{x}_{(1)}, \hat{\bm{u}}_{(1)}) + (1-\lambda) r(\bm{x}_{(2)}, \hat{\bm{u}}_{(2)})  +   \sum_{s=1}^N \sum_{i=1}^{M_{t+1}} p_s [\lambda\hat{\gamma}_{(1), s,i} +(1-\lambda) \hat{\gamma}_{(2),s,i} ] v (\bm{x}^{[i]})\\
& = \lambda (\hat{T}_t v) (\bm{x}_{(1)}) + (1-\lambda) (\hat{T}_t v) (\bm{x}_{(2)}),
\end{split}
\end{equation}
where the second inequality holds by the convexity of $r$.
This implies that $\hat{T}_t v$ is convex on $\mathcal{Z}_t$. \qed
\end{proof}

\bibliographystyle{IEEEtran}

\bibliography{reference}

\end{document}